\documentclass[12pt,abstracton]{scrartcl}
\usepackage[pdfpagelabels=true,unicode=true]{hyperref}
\hypersetup
{
	pdfauthor={Leandro Cioletti and Artur O. Lopes},
	pdftitle={Correlation Inequalities and Monotonicity Properties of the Ruelle Operator},
	pdfkeywords={Ruelle Operator, Gibbs Measures, Symbolic Dynamics, Correlation Inequalities},
	pdfsubject={Correlation inequalities for eigenmeasures},
	pdflang={en-US}
}
\usepackage{amsmath,amsfonts,amscd,amssymb,amsthm}
\usepackage{graphicx}
\usepackage[english]{babel}
\usepackage{mathrsfs}
\usepackage{subfig}
\newtheorem{theorem}{Theorem}
\newtheorem{definition}{Definition}

\newtheorem{corollary}{Corollary}
\newtheorem{lemma}{Lemma}
\newtheorem{proposition}{Proposition}
\theoremstyle{definition}
\newtheorem{example}{Example}
\newtheorem{remark}{Remark}

\title{\vspace*{-0.5cm}
\LARGE{
Correlation Inequalities and Monotonicity Properties of the Ruelle Operator}  }
\author{
  \large Leandro Cioletti
  \\[-0.3cm]
  \footnotesize Departamento de Matem\'atica - UnB
  \\[-0.3cm]
  \footnotesize 70910-900, Bras\'ilia, Brazil
  \\[-0.3cm]
  \footnotesize\texttt{cioletti@mat.unb.br}
  \and
  \large Artur O. Lopes
  \\[-0.3cm]
  \footnotesize Departamento de Matem\'atica - UFRGS
  \\[-0.3cm]
  \footnotesize 91509-900, Porto Alegre, Brazil
  \\[-0.3cm]
  \footnotesize\texttt{arturoscar.lopes@gmail.com}
}

\date{\small\today}

\begin{document}
\makeatletter
\def\blfootnote{\gdef\@thefnmark{}\@footnotetext}
\let\@fnsymbol\@roman
\makeatother

\maketitle
\begin{abstract}
Let $X = \{1,-1\}^\mathbb{N}$ be the symbolic space
endowed with a partial order $\succeq$, where $x \succeq y$, if $x_j\geq y_j$,
for all $j\in \mathbb{N}$.
A function $f:X \to \mathbb{R}$ is called increasing
if any pair $x,y\in X$, such that $x \succeq y$, we have
$f(x) \geq f(y).$
A Borel probability measure $\mu$ over $X$
is said to satisfy the FKG inequality if for any pair of continuous increasing functions
$f$ and $g$ we have $\mu(fg)-\mu(f)\mu(g)\geq 0$.
In the first part of the paper we prove the validity of the FKG inequality on
Thermodynamic Formalism setting for a class of eigenmeasures of the dual of
the Ruelle operator, including several examples of interest in Statistical Mechanics.
In addition to deducing this inequality  in cases not
covered by classical results about attractive specifications our
proof has advantage of to be easily adapted for suitable subshifts.
We review (and provide proofs in our setting)
some classical results about the long-range Ising model on the lattice $\mathbb{N}$
and use them to deduce some monotonicity properties of the associated Ruelle operator
and their relations with phase transitions.

As is widely known, for some continuous potentials
does not exists a positive continuous eigenfunction associated to the spectral radius
of the Ruelle operator acting on $C(X)$.
Here we employed some ideas related to the involution kernel in order to
solve the main eigenvalue problem in a suitable sense
for a class of potentials having low regularity.
From this we obtain an explicit tight upper bound for the main eigenvalue
(consequently for the pressure) of the Ruelle operator associated
to Ising models with $1/r^{2+\varepsilon}$ interaction energy.
Extensions of the Ruelle operator to suitable Hilbert Spaces
are considered and a theorem solving to the main eigenvalue
problem (in a weak sense) is obtained by using the Lions-Lax-Milgram theorem.
We generalize results due to P. Hulse  on  attractive $g$-measures.
We also present the graph of
the main eigenfunction in some examples - in some cases
the numerical approximation shows the evidence of not being continuous.
\end{abstract}

\blfootnote{The authors are supported by CNPq-Brazil.}

\section{Introduction}

The primary aim of this paper is to relate the Fortuin-Kasteleyn-Ginibre
(FKG) inequality to the study of the main eigenvalue problem for Ruelle operator associated
to an attractive potential $A$ having low regularity (meaning $A$
lives outside of the classical H\"older, Walters and Bowen spaces).

The FKG inequality \cite{MR0309498} is a strong correlation inequality and a fundamental
tool in Statistical Mechanics. An earlier version of this inequality for product measures
was obtained by Harris in \cite{MR0115221}. Holley in \cite{MR0341552}
generalized the FKG inequality in the context of finite distributive lattice.
In the context of Symbolic Dynamics the FKG inequality can be formulated
as follows.
Let us consider the symbolic space $X = \{1,-1\}^\mathbb{N}$ with an
additional structure which is a partial order $\succeq$,
where $x \succeq y$, if $x_j\geq y_j$, for all $j\in \mathbb{N}$.
A function $f:X \to \mathbb{R}$ is said \textit{increasing}
if for all $x,y\in X$, such that $x \succeq y$, we have
$f(x) \geq f(y).$
A Borel probability measure $\mu$ over $X$
will be said to satisfy the FKG inequality
if for any pair of continuous increasing functions
$f$ and $g$ we have
\[
\int_{X}fg\, d\mu -\int_{X}f\, d\mu \int_{X}g\, d\mu \geq 0.
\]
In Probability Theory such measure are sometimes called \textit{positively associated}.

Our first result asserts that for any potential
$A\in\mathcal{E}$ (Definition \ref{classe-E})
the probability measure defined in \eqref{gene}
satisfies the FKG inequality. As a consequence of this result
we are able to shown that at least one eigenprobability of $\mathscr{L}_{A}^*$,
associated to its spectral radius,
must satisfies the FKG inequality. Some similar results for $g$-measures
where obtained by P. Hulse in \cite{MR2218769,MR1488324,MR1101084}.
Potentials satisfying a condition similar to Definition \ref{classe-E}
are called \textit{attractive potentials} on these papers,
which is a terminology originated from
\textit{attractive specifications} sometimes used in Statistical Mechanics.

The class $\mathcal{E}$ includes
some interesting examples of potentials $A$ as the
ones described by (\ref{nonsimp}) and (\ref{simp}).

Establishing FKG inequality for  continuous potentials with low regularity is
a key step to study, for example, the Dyson model on the lattice
$\mathbb{N}$, within the framework of Thermodynamic Formalism.
A Dyson model (see \cite{MR0436850}) is a special long-range ferromagnetic
Ising model, commonly defined on the lattice $\mathbb{Z}$.
The Dyson model is a very important example in Statistical Mechanics
exhibiting the phase transition phenomenon in one dimension.
This model still is a topic of active research and currently it is being
studied in both lattices  $\mathbb{N}$ and $\mathbb{Z}$,
see the recent preprints \cite{van2016decimation,johansson2016phase}
and references therein. In both works whether the DLR-Gibbs measures associated
to the Dyson model is a $g$-measures is asked.

In \cite{johansson2016phase} the authors proved that the Dyson model
on the lattice $\mathbb{N}$ has phase transition. This result
is an important contribution to the Theory of Thermodynamic Formalism
since very few examples of phase transition on the lattice $\mathbb{N}$
are known (see \cite{MR1244665,MR3350377,MR2807681,MR0435352,MR2218769}).
In this work is also proved that the critical temperature
of the Dyson model on the lattice $\mathbb{N}$
is at most four times the critical temperature of Dyson model
on the lattice $\mathbb{Z}$. The authors also conjectured that the critical
temperature for both models coincides. We remark that
the explicit value of the critical temperature for the Dyson model
on both lattices still is an open problem. Moreover
there are very few examples in both Thermodynamic Formalism and Statistical Mechanics,
where the explicit value of the critical temperature is known.
A remarkable example where the critical temperatures is explicitly
obtained is the famous work by Lars Onsager \cite{MR0010315} and
the main idea behind this computation is the Transfer Operator.

Although the Ruelle operator $\mathscr{L}_{A}$ (associated to
the potential $A$) have been intensively studied,
since its creation in 1968, and became a key concept in Thermodynamic Formalism
a little is known about $\mathscr{L}_{A}$, when $A$ is the Dyson potential.
The difficult in using this operator to study the Dyson model
is the absence of positive continuous eigenfunctions associated to the spectral
radius of its action on $C(X)$. An alternative to overcome this problem
is to consider extensions of this operator to larger spaces than $C(X)$,
where a weak version of Ruelle-Perron-Frobenius theorem can be obtained.
We point out that continuous potentials may not have a continuous positive eigenfunction
but the dual of the Ruelle operator always has an eigenprobability.
Here we study the extension of the Ruelle operator
to the Lebesgue space $L^2(\nu_{A})$,
where $\nu_{A}$ is an eigenmeasure for $\mathscr{L}^{*}_{A}$.
We study the existence problem of the main eigenfunction
in such spaces by using the involution kernel and subsequently the
Lions-Lax-Milgram theorem.

In another direction we show how to use the involution kernel representation
of the main eigenfunction and the FKG inequality to obtain non-trivial
upper bound for the topological pressure of potentials of the form
\begin{equation} \label{nonsimp}
A(x)
=
a_0x_0x_1 + a_1x_0x_2+ a_2x_0x_3 +\ldots+ a_n x_0 x_{n+1}+\ldots
\end{equation}
which is associated to a long-range Ising model, when $(a_n)_{n\geq 1}$
is suitable chosen.
A particular interesting case occurs when $a_n=n^{-\gamma}$ with $\gamma>1$
(see end of section 5  in \cite{cioletti2014interactions} for the relation with
the classical Long-range Ising model interaction).

The above mentioned upper bound coincides with the topological pressure of
a \textit{product-type potential} $B$ (which is different but similar to the previous one)
given by
\begin{equation} \label{simp}
B(x)
=
a_0 \,x_0 + a_1\, x_1+ a_2\, x_2+\ldots+ a_n\, x_n+\ldots.
\end{equation}
See \cite{JLMS:JLMS12031} for the computation of the topological pressure of $B$.
In some sense we can think of this model as a simplified version of the previous one.
In this simpler model is possible to exhibit
explicit expressions for the eigenfunction and eigenprobability of the Ruelle
operator $\mathscr{L}_{B}$, see \cite{JLMS:JLMS12031}.

Suppose that  $a_n=n^{-\gamma}$ for all $n\geq 1$, for both potentials \eqref{nonsimp}
and \eqref{simp}. Although the potentials $A$ and $B$ have completely different
physical interpretations (two-body interactions versus self interaction)
from the Thermodynamic Formalism point of view they
have interesting similarities.
For example, in the simplified model (case (\ref{simp})) one can show that the  Ruelle operator
stops having positive continuous eigenfunction
if $\gamma\leq 2$ (see \cite{JLMS:JLMS12031}).
On the other hand, in a similar fashion,  for the potential $A$ in case (\ref{nonsimp}) and $\gamma<2$,
Figure \ref{fig5}  on section \ref{sim} - obtained via a numerical approximation -
seems to indicate that there exists a non-continuous eigenfunction.

.

When $\gamma>2$, the eigenfunctions associated to both potentials
are very well-behaved and they belong to the Walters space.
Although for $3/2<\gamma\leq 2$ we do not have phase transition for
the potential $B$ the unique non-negative
eigenfunction for $\mathscr{L}_{B}$ is such that
its values oscillate between zero and infinity
in any cylinder subset of $X$.
On the other hand, if $1<\gamma< 2$ then
we know from \cite{johansson2016phase} that
there is phase transition for the potential $A$ in the sense of the existence of two eigenprobabilities.
These observations suggest that the main eigenfunction
of $\mathscr{L}_{B}$ carries information about phase
transition for the potential $A$.

\bigskip

In Section \ref{inv} we show how to use the involution kernel
in order to construct an ``eigenfunction''
(the quotes is because of they are only defined on a dense subset of $X$)
for the Dyson model associated to the spectral radius of the Ruelle operator.

Some results of P. Hulse are generalized to non normalized potentials
in Section \ref{classF} and use some stochastic dominations coming from these extensions
to obtain uniqueness results for eigenprobabilities for a certain class of
potentials with low regularity.

\section{Increasing Functions and Correlation Inequalities}

Let $\mathbb{N}$ be the set of the non-negative integers  and  $a\in\mathbb{R}$
be any fixed positive number.
Consider the symbolic space $X = \{-a,a\}^\mathbb{N}$
and the left shift mapping $\sigma : X \to X $
defined for each $x\equiv (x_1,x_2,\ldots)$ by
$\sigma(x)= \sigma(x_0,x_1,x_2,\ldots) =  (x_1,x_2,x_3\ldots)$.
As usual we endow $X$ with its standard distance $d_X$,
where $d_{X}(x,y)= 2^{-N}$, where
$N =\inf \{i\in\mathbb{N}: x_i\neq y_i\}$.
As mentioned before we consider the partial order $\succeq$ in $X$,
where $x \succeq y$, iff $x_j\geq y_j$,
for all $j\in \mathbb{N}$. A function
$f:X \to \mathbb{R}$  is called increasing (decreasing)
if for all $x,y\in X$ such that $x \succeq y$, we have that
$f(x) \geq f(y)$ ($f(x) \leq f(y)$).
The set of all continuous increasing and decreasing functions
are denoted by $\mathcal{I}$ and $\mathcal{D}$, respectively.

For each $n\geq 1$,  $t\in\{-a,a\}$ and $x,y\in X$ will be
convenient in this section to use the following notations
\[
[x|y]_n \equiv (x_0,\ldots,x_{n-1},y_n,y_{n+1},\ldots)
\quad\text{and}\quad
[x|t|y]_n \equiv (x_0,\ldots,x_{n-1},t,y_{n+1},\ldots).
\]

A function $A:X\to\mathbb{R}$ will be called a potential.
For each potential $A$, $x\in X$ and $n\geq 1$ we define
$
S_n(A)(x)
\equiv
A(x)+\ldots +A\circ \sigma^{n-1}(x).
$

For any fixed $y\in X$ and $n\geq 1$ we define
a probability measure over $X$ by the following expression
\begin{equation}
\label{gene}
\mu_n^y
=
\!\!\!\!
\sum_{ x_0,\ldots,x_{n-1}=\pm a }
\!\!\!\!
\frac{\exp (S_n(A)([x|y]_n))}{Z_n^y}
\delta_{([x|y]_n)},
\ \text{where}\
Z_n^y
\equiv
\!\!\!\!
\sum_{ x_0,\ldots,x_{n-1}=\pm a }
\!\!\!\!
\exp (S_n(A)([x|y]_n))
\end{equation}
and $\delta_{x}$ is the Dirac measure supported on the point $x$.
The normalizing factor $Z_n^y$ is called \textit{partition function}
(associated to the potential $A$).

\begin{definition}
Let $\varepsilon>0$ be given.
A function
$\widetilde{A}:[-(a+\varepsilon),a+\varepsilon]^{\mathbb{N}}\to \mathbb{R}$
is called a differentiable
extension of a potential $A:X\to \mathbb{R}$ if for all
$x\in \{-a,a\}^{\mathbb{N}}$ we have
$\widetilde{A}(x)=A(x)$ and for all $j,n\in\mathbb{N}$
the following partial derivatives exist and the mappings
\[
\big( -(a+\varepsilon), a+\varepsilon  \big)
\ni
t
\mapsto
\frac{\partial \widetilde{A}}{\partial x_j}
(x_0,\ldots,x_{n-1},t,x_{n+1},\ldots)
\]
are continuous for any fixed $x\in [-a,a]^{\mathbb{N}}$.
\end{definition}

To avoid a heavy notation, a differentiable extension $\widetilde{A}$
of a potential $A$ will be simply denoted by $A$.
Note that the Ising type potentials are examples of continuous potentials
admitting natural differentiable extensions.

\begin{definition}[Class $\mathcal{E}$ potential] \label{classe-E}
We say that a continuous potential $A:X\to\mathbb{R}$
belongs to class $\mathcal{E}$ if it admits a differentiable
extension satisfying:
for any fixed $t\in [-a,a]$, $y\in [-a,a]^{\mathbb{N}}$,  $n \geq 1$
we have that
\begin{align}\label{khs1}
(x_0,x_1,\ldots )
&\longmapsto
\frac{d}{ dt} S_{n}(A)([x|t|y]_n),
\end{align}
is an increasing function from $X$ to $\mathbb{R}$.
\end{definition}

Let $n\geq 1$ be fixed and $f,g:X\to\mathbb{R}$ two real increasing functions,
with respect to the partial order $\succeq$, depending
only on its first $n$ coordinates.
The main result of the next section states that for all
potential $A$ in the class $\mathcal{E}$ the probability measure
$\mu_n^y$ given by \eqref{gene} satisfies the \textbf{FKG Inequality}
\begin{align}\label{def-corr-ineq}
\int_{X} f g \,d \mu_n^{y} -
\int_{X} f\, d\mu_n^{y}
\
\int_{X} g\, d\mu_n^{y}
\geq
0
\end{align}
for any choice of $y\in X$.

\begin{remark}
If for all $n\geq 1$ the probability measure $\mu_{n}^{y} $
satisfies \eqref{def-corr-ineq}
and $\mu_n^{y}\rightharpoonup \mu$ then $\mu$
satisfies \eqref{def-corr-ineq}.
\end{remark}

In what follows we exhibit explicit examples of potentials
in the class $\mathcal{E}$.

\subsection{Dyson Potential}

An Ising type model, on the lattice $\mathbb{N}$, in Statistical Mechanics
is a model defined in the symbolic space $X$, with $a=1$.
Here we call a \textbf{Ising type potential} any real function $A:X\to\mathbb{R}$
of the form
$
A(x)
=
hx_0
+x_0\sum_{i} a_ix_i,
$
where $x\in \{-1,1\}^{\mathbb{N}}$ and $h,a_1,a_2,\ldots$ are
fixed real numbers satisfying  $\sum_n |a_n|< \infty$.
An interesting family of such potentials is given by
\begin{align}\label{def-potencial-dyson}
A(x)
\equiv
hx_0+ x_0 x_1 + \frac{x_0 x_2}{2^{\alpha}} + \frac{x_0 x_3}{3^{\alpha}}+\ldots
,\quad \text{where}\ \alpha>1\ \text{and}\ h\in\mathbb{R}
\end{align}
When $1<\alpha<2$ the potential $A$
is sometimes called \textbf{Dyson potential}.
It is worth to mention that a Dyson potential
is not an increasing, decreasing or H\"older function.

As will be shown latter the probability measure $\mu_n^y$ determined by
the potential $A$ given in \eqref{def-potencial-dyson} satisfies the
correlation inequality, of the last section, \textbf{for any} choice of $y\in X$.

\bigskip

\noindent\textbf{Remark.}
The Dyson potential for any
fixed $h\in \mathbb{R} $ and $1<\alpha<2$ belongs to the class $\mathcal{E}$.
Indeed, a straightforward computation shows that
\[
S_n(A)([x|t|y]_n)
=
x_0\, t \,n^{-\alpha}+
x_1\, t \,(n-1)^{-\alpha}+\ldots +
x_{n-1} \, t
+
D_n,
\]
for some constant $D_n$, which depends on $x$ and $y$ but not on $t$.
From this expression the condition \eqref{khs1} can be
immediately verified. More generally, any \textbf{Ising type potential}
with $a_n\geq 0$ for all $n\geq1$, satisfies the hypothesis of Theorem \ref{puu}.
In this case, the potential $A$ is sometimes called ferromagnetic potential.

\subsection{The FKG Inequality}

The results obtained in this section are inspired in the proof
of the FKG inequality for ferromagnetic Ising models
presented in \cite{MR2189669}. In this reference the inequality
is proved under assumptions on the local behavior of the interactions
of the Ising model but here our hypothesis are
about the global behavior of the potential.
For sake of simplicity we assume that $a=1$.
The arguments and results obtained here can be immediately
generalized for any other choice of $a>0$.

In order to keep the paper self-contained we recall the
following classical result.
\begin{lemma}\label{lema-FKG-n=1}
Let $E\subset \mathbb{R}$ and
$(E,\mathscr{F},\lambda)$ a probability space.
If $f,g:E\to\mathbb{R}$ are increasing functions then
\begin{align}\label{FKG-n=1}
\int_{E} fg\ d\lambda
\geq
\int_{E} f\ d\lambda
\int_{E} g\ d\lambda.
\end{align}
\end{lemma}

\begin{proof}
Since $f$ and $g$ are increasing functions, then for any pair
$(s,t)\in E\times E$ we have $0\leq [f(s)-f(t)][g(s)-g(t)]$.
By integrating both sides of this inequality, with respect to the product measure
$\lambda\times\lambda$, using the elementary properties of the integral
and $\lambda$ is a probability measure we finish the proof.
\end{proof}

Now we present an auxiliary combinatorial lemma that
will be used in the proof of Theorem \ref{puu}.

\begin{lemma}\label{lema-mu-lambda}
Let $E=\{-1,1\}$, $y\in X$ fixed and $f:X\to\mathbb{R}$ a continuous function.
Then the following identity holds for all $n\geq 1$
\[
\int_{E}
\left[
	\int_{X} f\, d\mu_{n}^{[y|t|y]_n}
\right]
d\lambda(t)
=
\int_{X} f\, d\mu_{n+1}^{y},
\quad\text{where}\
\lambda \equiv
\sum_{t=\pm 1}
\frac{Z_{n}^{[y|t|y]_n}}{Z_{n+1}^{y}}\delta_{t}.
\]
\end{lemma}

\begin{proof}
By using the definitions of $\lambda$ and $\mu_{n}^{y}$, respectively
we get
\begin{align*}
\int_{E}
\left[
	\int_{X} f\, d\mu_{n}^{[y|t|y]_n}
\right]
d\lambda(t)
&=
\sum_{t=\pm 1}
\frac{Z_{n}^{[y|t|y]_n}}{Z_{n+1}^{y}}
\int_{X} f\, d\mu_{n}^{[y|t|y]_n}
\\
&=
\sum_{t=\pm 1}
\frac{Z_{n}^{[y|t|y]_n}}{Z_{n+1}^{y}}
\sum_{x_0,\ldots,x_{n-1}=\pm 1}
f([x|t|y]_n) \frac{\exp( S_n(A)([x|t|y]_n) )}{Z_{n}^{[y|t|y]_n}}
\\
&=
\sum_{t=\pm 1}
\frac{1}{Z_{n+1}^{y}}
\sum_{x_0,\ldots,x_{n-1}=\pm 1}
f([x|t|y]_n) \exp( S_n(A)([x|t|y]_n) )
\\
&=
\frac{1}{Z_{n+1}^{y}}
\sum_{t=\pm 1}\ \
\sum_{x_0,\ldots,x_{n-1}=\pm 1}
f([x|t|y]_n) \exp( S_n(A)([x|t|y]_n) )
\\
&=
\frac{1}{Z_{n+1}^{y}}
\sum_{x_0,\ldots,x_{n}=\pm 1}
f([x|y]_n) \exp( S_n(A)([x|y]_n))
\\
&=
\int_{X} f\, d\mu_{n+1}^{y}.
\end{align*}
\end{proof}

To shorten the notation in the remaining of this section,
we define for each $n\geq 1$, $x\in X$ and $t\in [-1,1]$
the following  weights
\begin{align}\label{rir1}
W_n([x|t|y]_n)
&\equiv
\frac{
		\exp( S_{n}(A)([x|t|y]_n)
	}
	{
		\displaystyle
		\sum_{z_0,z_1,\ldots,z_{n-1}=\pm 1}
		\!\!\!\!\!\!\!\!
		\exp( S_{n}(A)([z|t|y]_n) 	)
	}
=
\frac{
		\exp( S_{n}(A)([x|t|y]_n)
	}
	{Z_n^{[y|t|y]_n}	}
\end{align}

\begin{lemma}\label{lema-principal-FKG}
Let $n\geq 1$ and $y\in X$ be fixed, $f:X\to\mathbb{R}$ an increasing function,
depending only on its first $n$ coordinates $(x_0,\ldots,x_{n-1})$.
If the potential $A$ belongs to the class $\mathcal{E}$ and $\mu_n$
satisfies the inequality \eqref{def-corr-ineq}, then
\begin{align}\label{eq-int-t-bem-definida}
[-1,1]\ni t\mapsto
\int_{X} f\, d\mu_{n}^{[y|t|y]_n}
\end{align}
is a real increasing function.
\end{lemma}

\begin{proof}
We first observe that the integral in \eqref{eq-int-t-bem-definida} is
well-defined because $A$ admits a differentiable extension defined in whole
space $[-(1+\varepsilon),1+\varepsilon]^{\mathbb{N}}$.

By using that $f$ depends only on its first $n$ coordinates we have
the following identity for any $y\in X$
\begin{align*}
\int_{X} f\, d\mu_{n}^{[y|t|y]_n}
&=
\sum_{x_0,\ldots,x_{n-1}=\pm 1}
\!\!\!\!
f([x|t|y]_n)W_n([x|t|y]_n)
=
\sum_{x_0,\ldots,x_{n-1}=\pm 1}
\!\!\!\!
f([x|y]_n)W_n([x|t|y]_n).
\end{align*}
Since $A$ belongs to the class $\mathcal{E}$
follows from the expression \eqref{rir1} that $W_n([x|t|y]_n)$ has
continuous derivative and therefore to prove the lemma is enough to prove that
\begin{align}\label{eq-aux1-Lema-pesos-crescentes}
\frac{d}{dt} \int_{X} f\, d\mu_{n}^{[y|t|y]_n}
=
\sum_{x_0,\ldots,x_{n-1}=\pm 1}
f([x|y]_n) \frac{d}{dt}W_n([x|t|y]_n)
\end{align}
in non-negative.

By using the quotient rule we get that the derivative appearing
in the above expression is equal to
\begin{align}\label{eq-aux2-Lema-pesos-crescentes}
\frac{d}{dt}W_n([x|t|y]_n)
&=
\frac{d}{dt}\frac{\exp( S_{n}(A)([x|t|y]_n)}{Z_n^{[y|t|y]_n}}
\nonumber
\\[0.3cm]
&=
\frac{\exp( S_{n}(A)([x|t|y]_n)}{Z_n^{[y|t|y]_n}}
\left[\frac{d}{dt} S_{n}(A)([x|t|y]_n)
-
\frac{1}{Z_n^{[y|t|y]_n}}
\frac{d}{dt} Z_n^{[y|t|y]_n}
\right].
\nonumber
\\[0.3cm]
&=
W_n([x|t|y]_n)
\left[\frac{d}{dt} S_{n}(A)([x|t|y]_n)
-
\frac{1}{Z_n^{[y|t|y]_n}}
\frac{d}{dt} Z_n^{[y|t|y]_n}
\right].
\end{align}

Note that the last term in the rhs above is equal to
\begin{align}\label{eq-aux3-Lema-pesos-crescentes}
\frac{1}{Z_n^{[y|t|y]_n}}
\frac{d}{dt} Z_n^{[y|t|y]_n}
&=
\frac{1}{Z_n^{[y|t|y]_n}}
\frac{d}{dt} \sum_{x_0,\ldots x_{n-1}=\pm 1} \exp(S_n(A)([x|t|y]_n))
\nonumber
\\
&=
\frac{1}{Z_n^{[y|t|y]_n}}
\sum_{x_0,\ldots x_{n-1}=\pm 1}
\exp(S_n(A)([x|t|y]_n))
\frac{d}{dt} S_n(A)([x|t|y]_n)
\nonumber
\\
&=
\int_{X}
\frac{d}{dt} S_n(A)([x,t,y]_n)
\, d\mu_{n}^{[x|t|y]}(x).
\end{align}

Replacing the expression \eqref{eq-aux3-Lema-pesos-crescentes}
in \eqref{eq-aux2-Lema-pesos-crescentes} we get
\begin{align*}
\frac{d}{dt}W_n([x|t|y]_n)
=
W_n([x|t|y]_n)
\left[\frac{d}{dt} S_{n}(A)([x|t|y]_n)
-
\int_{X}
\frac{d}{dt} S_n(A)([x,t,y]_n)
\, d\mu_{n}^{[x|t|y]}(x)
\right].
\end{align*}
By replacing the above expression in
\eqref{eq-aux1-Lema-pesos-crescentes}
we obtain
\begin{align*}
\frac{d}{dt} \int_{X} f\, d\mu_{n}^{[y|t|y]_n}
&=
\int_{X} f(x) \frac{d}{dt} S_{n}(A)([x,t,y]_n)\, d\mu_{n}^{[x|t|y]_n}(x)
\\
&\qquad \quad -
\int_{X} f \, d\mu_{n}^{[x|t|y]_n}
\int_{X} \frac{d}{dt} S_{n}(A)([x,t,y]_n)\, d\mu_{n}^{[x|t|y]_n}(x)
\end{align*}
which is non-negative because $f$ is increasing
$A\in\mathcal{E}$ and the probability measure $\mu_n$
satisfies the inequality \eqref{def-corr-ineq} by hypothesis.
\end{proof}

\begin{theorem} \label{puu}
Let $A:X\to\mathbb{R}$ be a potential in the class $\mathcal{E}$.
For any fixed $y\in [-1,1]^{\mathbb{N}}$ and for all $n\geq 1$
the probability measure

\begin{equation}\label{puts1}
\mu_n^y
=
\sum_{ x_0,\ldots,x_{n-1}=\pm 1  }
\frac{\exp(S_n(A)([x|y]_n)}{Z_n^{y}}
\delta_{([x|y]_n)},
\end{equation}
where $Z_n^{y}$ is the standard partition function,
satisfies the correlation inequality \eqref{def-corr-ineq}.
\end{theorem}

\begin{proof}

The proof is by induction in $n$. The inequality
\eqref{def-corr-ineq}, for $n=1$, follows from
a straightforward application of Lemma \ref{lema-FKG-n=1}.
Indeed, for any fixed $y\in X$ the mappings
$X\ni x\longmapsto f(x_0,y_1,y_2,\ldots)$ and
$X\ni x\longmapsto g(x_0,y_1,y_2,\ldots)$ are clearly increasing.
By thinking of these maps as functions from $E=\{-1,1\}$ to $\mathbb{R}$
and $\mu_1^{y}$ as a probability measure over $E$,
we can apply Lemma \ref{lema-FKG-n=1} to get the conclusion.

The induction hypothesis is formulated as follows. For some $n\geq 2$ assume
that for all $y\in X$ and any pair of real continuous increasing functions
$f$ and $g$, depending only on its first $n$ coordinates,  we have
\[
\int_X fg\, d \mu^y_{n}
\geq
\int_{X} f\, d \mu^{y}_{n}
\int_{X} g\, d \mu^{y}_{n}.
\]

Now we prove that $\mu_{n+1}^y$ satisfy \eqref{def-corr-ineq}.
From the definition we have that
\begin{align*}
\int_{X}fg \, d \mu^y_{n+1}
&=
\sum_{x_0,\ldots,x_{n}=\pm 1}
f([x|y]_n)g([x|y]_n)
\frac{\exp(S_n(A)([x|y]_n)}{Z_n^{y}}
\\[0.3cm]
&=
\frac{Z_{n}^{[y|1|y]_n}}{Z_{n+1}^{y}}
\sum_{x_0,\ldots,x_{n-1}=\pm 1}
\!\!\!\!
W_n([x|1|y]_n)
f([x|1|y]_n)
g([x|1|y]_n)+
\\[0.1cm]
&\qquad
\frac{Z_{n}^{[y|-1|y]_n}}{Z_{n+1}^{y}}
\sum_{x_0,\ldots,x_{n-1}=\pm 1}
\!\!\!\!
W_n([x|-1|y]_n)
f([x|-1|y]_n)
g([x|-1|y]_n)
\\[0.3cm]
&=
\frac{Z_{n}^{[y|1|y]_n}}{Z_{n+1}^{y}}
\int_{X} fg\, d\mu_{n}^{[y|1|y]_n}
+
\frac{Z_{n}^{[y|-1|y]_n}}{Z_{n+1}^{y}}
\int_{X} fg\, d\mu_{n}^{[y|-1|y]_n}
\end{align*}
By using the induction hypothesis on both terms
in the rhs above we get that
\begin{align}\label{eq-aux-FKG1}
\int_{X} fg \, d \mu^y_{n+1}
&\geq
\sum_{t=\pm 1}
\frac{Z_{n}^{[y|t|y]_n}}{Z_{n+1}^{y}}
\int_{X} f\, d\mu_{n}^{[y|t|y]_n}
\int_{X} g\, d\mu_{n}^{[y|t|y]_n}
\nonumber
\\
&\equiv
\int_{E}
\left[
	\left( \int_{X} f\, d\mu_{n}^{[y|t|y]_n} \right)
	\left( \int_{X} g\, d\mu_{n}^{[y|t|y]_n} \right)
\right]
d\lambda(t),
\end{align}
where $E=\{-1,1\}$ and $\lambda$ is defined as in Lemma \ref{lema-mu-lambda}.
Since $Z_{n}^{[y|1|y]_n}+Z_{n}^{[y|-1|y]_n}= Z_{n+1}^{y}$
follows that $\lambda$ is a probability measure over $E$.
From Lemma \ref{lema-principal-FKG} we get that both functions
\[
t\mapsto
\int_{X} f\, d\mu_{n}^{[y|t|y]_n}
\quad\text{and}\quad
t\mapsto
\int_{X} g\, d\mu_{n}^{[y|t|y]_n}
\]
are increasing functions.
To finish the proof it is enough to
apply Lemma \ref{lema-FKG-n=1} to the rhs of \eqref{eq-aux-FKG1}
obtaining
\begin{align*}
\int_{X} fg \, d \mu^y_{n+1}
&\geq
\int_{E}
\left[
	\int_{X} f\, d\mu_{n}^{[y|t|y]_n}
\right]
d\lambda(t)
\int_{E}
\left[
	\int_{X} g\, d\mu_{n}^{[y|t|y]_n}
\right]
d\lambda(t)
\\
&=
\int_{X} f \, d \mu^y_{n+1}
\int_{X} g \, d \mu^y_{n+1},
\end{align*}
where the last equality is ensured by the Lemma \ref{lema-mu-lambda}.
\end{proof}

\subsection{FKG Inequality and the Ising Model} \label{FKGCL}

In this section we recall the classical FKG inequality for the Ising model
as well as some of its applications. For more details
see \cite{MR2189669,MR0309498} and \cite{MR2108619}.

Let $\pmb{h}=(h_i)_{i\in\mathbb{N}} \in \ell^{\infty}(\mathbb{N})$
and $\pmb{J}\equiv \{J_{ij}\in \mathbb{R}: i,j\in\mathbb{N}\ \text{and}\ i\neq j \}$
be a collection of real numbers
belonging to the set
\begin{align}\label{condicao-regularidade}
\mathscr{R}(\mathbb{N})
=
\left\{
\pmb{J}:
\sup_{i\in\mathbb{N}}
\sum_{j\in\mathbb{N}\setminus\{i\}}
|J_{ij}|
<
+\infty
\right\}.
\end{align}
For each $n\in\mathbb{N}$ we define a real function
$
H_n:
X\times X\times \mathscr{R}(\mathbb{N})\times \ell^{\infty}(\mathbb{N})
\to
\mathbb{R}
$
by following expression
\begin{align}\label{hamiltoniano-Ising-geral}
H_n(x,y,\pmb{J},\pmb{h})=
\sum_{0\leq i<j\leq n-1} J_{ij}x_ix_j
+
\sum_{0\leq i\leq  n-1} h_ix_i
+
\sum_{ \substack{0\leq i\leq n-1 \\ j\geq n} } J_{ij}x_iy_j.
\end{align}
Note that the summability condition in \eqref{condicao-regularidade}
ensures that the series appearing in \eqref{hamiltoniano-Ising-geral} is absolutely
convergent and therefore $H_n$ is well defined.

For each $n\geq 1$, $y\in X$ and
$(\pmb{J},\pmb{h})\in \mathscr{R}(\mathbb{N})\times \ell^{\infty}(\mathbb{N})$
we define a probability measure by the following expression
\begin{align}\label{mu-n}
\mu_n^{y,\pmb{J},\pmb{h}}
=
\frac{1}{Z_n^{y,\pmb{J},\pmb{h}}}
\sum_{x_0,\ldots,x_{n-1}=\pm 1}
\exp(H_n(x,y,\pmb{J},\pmb{h}))
\delta_{([x|y]_n)},
\end{align}
where $Z_n^{y,\pmb{J},\pmb{h}}$ is the partition function.
In the next section we show that for suitable
choices of $\pmb{J}$ and $\pmb{h}$ the expression
\eqref{mu-n} can be rewritten in terms of the Ruelle
operator.

\begin{theorem}[FKG-Inequality]\label{teo-FKG-Mec-Est-Formulation}
Let $n\geq 1$, $\pmb{h}\in\ell^{\infty}(\mathbb{N})$
and $\pmb{J}\in\mathscr{R}(\mathbb{N})$ so that $J_{ij}\geq 0$
for any pair $i,j$.
If $f,g:X\to \mathbb{R}$ are increasing functions depending only
on its first $n$ coordinates, then
\[
\int_{X} fg\, d\mu_n^{y,\pmb{J},\pmb{h}}
-
\int_{X} f\, d\mu_n^{y,\pmb{J},\pmb{h}}
\int_{X} g\, d\mu_n^{y,\pmb{J},\pmb{h}}
\geq
0.
\]
\end{theorem}

\begin{proof}
We can prove this theorem using the same ideas
employed in the proof of Theorem \ref{puu}.
For details, see \cite{MR2189669}.
\end{proof}

Note that the Hamiltonian
$
H_n:
X\times X\times \mathscr{R}(\mathbb{N})\times \ell^{\infty}(\mathbb{N})
\to
\mathbb{R}
$
admits a natural differentiable extension to a function
defined on
$
\mathbb{R}^{\mathbb{N}}
\times
\mathbb{R}^{\mathbb{N}}
\times
\mathscr{R}(\mathbb{N})
\times
\ell^{\infty}(\mathbb{N})
$
and so for any $f$, depending on its first $n$ coordinates,
the following partial derivatives exist and are
continuous functions
\[
\frac{\partial }{\partial h_j}
\int_{X} f\, d\mu_n^{y,\pmb{J},\pmb{h}},
\quad
\frac{\partial }{\partial y_j}
\int_{X} f\, d\mu_n^{y,\pmb{J},\pmb{h}}
\quad\text{and}\quad
\frac{\partial}{\partial J_{ij}}
\int_{X} f\, d\mu_n^{y,\pmb{J},\pmb{h}}
\]
\begin{corollary}\label{cor-derivada-gamma}
Under the hypothesis of Theorem \ref{teo-FKG-Mec-Est-Formulation}
we have
\[
\frac{\partial}{\partial h_i}
\int_{X} f\, d\mu_n^{y,\pmb{J},\pmb{h}}
=
\int_{X} f(x)x_i\, d\mu_n^{y,\pmb{J},\pmb{h}}(x)
-
\int_{X} f(x)\, d\mu_n^{y,\pmb{J},\pmb{h}}(x)
\int_{X} x_i\, d\mu_n^{y,\pmb{J},\pmb{h}}(x)
\geq
0.
\]
In particular, if $\widetilde{\pmb{h}}\succeq \pmb{h}$ then
\[
\int_{X} f\, d\mu_n^{y,\pmb{J},\widetilde{\pmb{h}}}
\geq
\int_{X} f\, d\mu_n^{y,\pmb{J},\pmb{h}}.
\]
\end{corollary}

\begin{corollary}\label{cor-monotonicidade-condicao-fronteira}
Under the hypothesis of Theorem \ref{teo-FKG-Mec-Est-Formulation}
if $x\succeq y$ then
\[
\int_{X} f\, d\mu_n^{x,\pmb{J},\pmb{h}}
\geq
\int_{X} f\, d\mu_n^{y,\pmb{J},\pmb{h}}.
\]
\end{corollary}

\begin{proof}
By considering the natural differentiable extension of $H_n$ to
$
\mathbb{R}^{\mathbb{N}}
\times
\mathbb{R}^{\mathbb{N}}
\times
\mathscr{R}(\mathbb{N})
\times
\ell^{\infty}(\mathbb{N})
$
we can proceed as in \eqref{eq-aux2-Lema-pesos-crescentes} obtaining
\begin{align*}
\frac{\partial}{\partial y_i}
\int_{X} f\, d\mu_n^{y,\pmb{J},\pmb{h}}
&=
\int_{X} f(x)\cdot \frac{\partial}{\partial y_i}
H_n(x,y,\pmb{J},\pmb{h})\, d\mu_n^{y,\pmb{J},\pmb{h}}(x)
\\[0.3cm]
&\qquad -
\int_{X} f(x)\, d\mu_n^{y,\pmb{J},\pmb{h}}(x)
\int_{X} \frac{\partial}{\partial y_i}
H_n(x,y,\pmb{J},\pmb{h})\, d\mu_n^{y,\pmb{J},\pmb{h}}(x).
\end{align*}
By using that $J_{ij}\geq 0$
we get from \eqref{hamiltoniano-Ising-geral} that
the mapping
$
x\longmapsto
(\partial/\partial y_i)
H_n(x,y,\pmb{J},\pmb{h})
$
is an increasing function. So we can apply the FKG inequality to
the rhs above to ensure that function
\[
y\longmapsto \int_{X} f\, d\mu_n^{y,\pmb{J},\pmb{h}}
\]
is coordinate wise increasing and therefore the result follows.
\end{proof}

To lighten the notation $\mu_n^{y,\pmb{J},\pmb{h}}$,
when $y=(1,1,1,\ldots)\equiv 1^{\infty}$ or similarly $y=-1^{\infty}$,
we will simply write $\mu_n^{+,\pmb{J},\pmb{h}}$ or
$\mu_n^{-,\pmb{J},\pmb{h}}$, respectively. If the parameters
$\pmb{J}$ and $\pmb{h}$ are clear from the context they will be omitted.

\begin{corollary}
Under the hypothesis of Theorem \ref{teo-FKG-Mec-Est-Formulation}
we have
\[
\int_{X} f\, d\mu_{n}^{+,\pmb{J},\pmb{h}}
\leq
\int_{X} f\, d\mu_{n-1}^{+,\pmb{J},\pmb{h}}
\qquad \text{and}\qquad
\int_{X} f\, d\mu_{n-1}^{-,\pmb{J},\pmb{h}}
\leq
\int_{X} f\, d\mu_{n}^{-,\pmb{J},\pmb{h}}
\]
\end{corollary}
\begin{proof}
The proof of these inequalities are similar, so it is enough to present the argument for
the first one. From Corollary \ref{cor-derivada-gamma} we get that
\[
\int_{X} f\, d\mu_{n}^{+,\pmb{J},\pmb{h}}
\leq
\lim_{h_n\to\infty}
\int_{X} f\, d\mu_{n}^{+,\pmb{J},\pmb{h}}.
\]
By using the definition of $\mu_{n}^{+,\pmb{J},\pmb{h}}$
we have
\begin{align}\label{mu-limite-campo}
\int_{X} f\, d\mu_{n}^{+,\pmb{J},\pmb{h}}
=
\sum_{t=\pm 1}
\frac{Z_{n-1}^{[1^{\infty}|t|1^{\infty}]_n,\pmb{J},\pmb{h}}}
{Z_{n}^{1^{\infty},\pmb{J},\pmb{h}}}
\!\!\!\!\!
\sum_{x_0,\ldots,x_{n-1}=\pm 1}
\!\!\!\!\!\!
f([x|t|1^{\infty}]_n)
\frac{\exp( H_n(x,[1^{\infty}|t|1^{\infty}]_n ,\pmb{J},\pmb{h}) )}
{Z_{n-1}^{[1^{\infty}|t|1^{\infty}]_n,\pmb{J},\pmb{h}}}
\end{align}
A straightforward computation shows that
\[
\lim_{h_n\to\infty}
\frac{Z_{n-1}^{[1^{\infty}|1|1^{\infty}]_n,\pmb{J},\pmb{h}}}
{Z_{n}^{1^{\infty},\pmb{J},\pmb{h}}}
=1
\quad\text{and}\quad
\lim_{h_n\to\infty}
\frac{Z_{n-1}^{[1^{\infty}|-1|1^{\infty}]_n,\pmb{J},\pmb{h}}}
{Z_{n}^{1^{\infty},\pmb{J},\pmb{h}}}
=0.
\]
The expression below is clearly uniformly bounded away from zero
and infinity, when $h_n$ goes to infinity
\[
\frac{\exp( H_n(x,[1^{\infty}|-1|1^{\infty}]_n ,\pmb{J},\pmb{h}) )}
{Z_{n-1}^{[1^{\infty}|-1|1^{\infty}]_n,\pmb{J},\pmb{h}}}.
\]
Finally by using l'hospital rule one can see that
\[
\lim_{h_n\to\infty}
\frac{\exp( H_n(x,[1^{\infty}|1|1^{\infty}]_n ,\pmb{J},\pmb{h}) )}
{Z_{n-1}^{[1^{\infty}|-1|1^{\infty}]_n,\pmb{J},\pmb{h}}}
=
\frac{\exp( H_{n-1}(x,[1^{\infty}|1|1^{\infty}]_n ,\pmb{J},\pmb{h}) )}
{Z_{n-1}^{[1^{\infty}|1|1^{\infty}]_n,\pmb{J},\pmb{h}}}.
\]
Piecing the last four observations together, we have
\begin{align*}
\lim_{h_n\to\infty}
\int_{X} f\, d\mu_{n}^{+,\pmb{J},\pmb{h}}
=
\!\!\!\!
\sum_{x_0,\ldots,x_{n-1}=\pm 1}
\!\!\!\!\!\!
f([x|1|1^{\infty}]_n)
\frac{\exp( H_{n-1}(x,[1^{\infty}|1|1^{\infty}]_n ,\pmb{J},\pmb{h}) )}
{Z_{n-1}^{[1^{\infty}|1|1^{\infty}]_n,\pmb{J},\pmb{h}}}
=
\mu_{n-1}^{+,\pmb{J},\pmb{h}}(f).
\end{align*}

\end{proof}

\begin{corollary}\label{cor-dom-estocastica-mu-finite-vol}
Under the hypothesis of Theorem \ref{teo-FKG-Mec-Est-Formulation}
we have
\[
\int_{X} f\, d\mu_{n-1}^{-,\pmb{J},\pmb{h}}
\leq
\int_{X} f\, d\mu_{n}^{-,\pmb{J},\pmb{h}}
\leq
\int_{X} f\, d\mu_{n}^{x,\pmb{J},\pmb{h}}
\leq
\int_{X} f\, d\mu_{n}^{+,\pmb{J},\pmb{h}}
\leq
\int_{X} f\, d\mu_{n-1}^{+,\pmb{J},\pmb{h}}.
\]
\end{corollary}
\begin{proof}
These four inequalities follows immediately from the
two previous corollaries.
\end{proof}

\section{Ruelle Operator}

We denote by $C(X)$ the set of all real continuous functions and
consider the Banach space $(C(X),\|\cdot\|_{\infty})$.
Given a continuous potential $A: X \to \mathbb{R} $ we
define the Ruelle operator $ \mathscr{L}_A:C(X)\to C(X)$
as being the positive linear operator sending
$ f \longmapsto \mathscr{L}_A (f)$, where for each $x\in X$
$$
\mathscr{L}_A (f)(x)
\equiv
\sum_{i\in \{-1,1\}} e^{ A(i\, ,x_0,x_1,x_2,...)} f(i\, ,x_0,x_1,x_2,...) .
$$

Let $\lambda_{A}$ denote the spectral radius of $\mathscr{L}_A$
acting on $(C(X),\|\cdot\|_{\infty})$.
If $A$ is a continuous potential, then
there always exists a Borel probability $\nu_{A}$ defined over $X$
such $\mathscr{L}_{A}^*(\nu_{A})=\lambda_{A} \, \nu_{A}$,
where $\mathscr{L}_{A}^*$ is the dual operator of the Ruelle operator.
We refer to any such $\nu_{A}$ as an eigenprobability for the potential $A$.

\begin{proposition}\label{prop-operador-medida-finite-vol}
Let $n\geq 1$ and $\pmb{J}\in \mathscr{R}(\mathbb{N})$
such that $J_{ij}=a_{|i-j|}\geq 0$, for some sequence $(a_n)_{n\geq 1}$
and	$\pmb{h}\in \ell^{\infty}(\mathbb{N})$ such that $h_i=h$ for all
$i\in\mathbb{N}$. Consider the potential $A:X\to\mathbb{R}$ given by
$A(x) = hx_0+x_0\sum_{n}a_nx_n$. Then for all $x,y\in X$ we have
$
H_{n}(x,y,\pmb{J},\pmb{h})
=
S_{n}(A)([x|y]_n)
$
and therefore for all continuous $f:X\to\mathbb{R}$ we have
\[
\frac{\mathscr{L}_{A}^n(f)(\sigma^n y)}
{\mathscr{L}_{A}(1)(\sigma^n y)}
=
\int_{X} f\, d\mu_{n}^y
=
\int_{X} f\, d\mu_n^{y,\pmb{J},\pmb{h}}.
\]
\end{proposition}

\begin{proof}
We first observe that the hypothesis $\pmb{J}\in \mathscr{R}(\mathbb{N})$
guarantee that the potential $A$ is well-defined since
its expression is given by an absolutely convergent series,
for any $x\in X$.  By rearranging the terms in the sum $S_{n}(A)([x|y]_n)$
it is easy to check that we can end up in $H_{n}(x,y,\pmb{J},\pmb{h})$.
Note that the translation invariance hypothesis placed in $J_{ij}$ is crucial for
validity of the previous statement. The above equation follows
directly from the equality
$
H_{n}(x,y,\pmb{J},\pmb{h})
=
S_{n}(A)([x|y]_n)
$
and definitions of such measures.
\end{proof}

\begin{corollary}\label{cor-monotonicidade-quociente-ruelles}
If $A(x) = hx_0+x_0\sum_{n}a_nx_n$, where $a_n\geq 0$ for all $n\geq 1$
and $\sum_{n}a_n<\infty$ then
	\[
	\frac{\mathscr{L}^{n-1}_{A}(f)(-1^{\infty})}
	{\mathscr{L}^{n-1}_{A}(1)(-1^{\infty})}
	\leq	
	\frac{\mathscr{L}^n_{A}(f)(-1^{\infty})}
	{\mathscr{L}^n_{A}(1)(-1^{\infty})}
	\leq	
	\frac{\mathscr{L}^n_{A}(f)(\sigma^n (x))}
	{\mathscr{L}^n_{A}(1)(\sigma^n (x))}
	\leq	
	\frac{\mathscr{L}^n_{A}(f)(1^{\infty})}
	{\mathscr{L}^n_{A}(1)(1^{\infty})}
	\leq	
	\frac{\mathscr{L}^{n-1}_{A}(f)(1^{\infty})}
	{\mathscr{L}^{n-1}_{A}(1)(1^{\infty})}.
	\]
\end{corollary}

\begin{proof}
This is straightforward application of the
Proposition \ref{prop-operador-medida-finite-vol} and
Corollary \ref{cor-dom-estocastica-mu-finite-vol}.
\end{proof}

If $A$ is a potential of the form
$A(x)= hx_0+x_0\sum_{n}a_nx_n$, where $a_n\geq 0$ and
$\sum_{n}a_n<\infty$,
then the above corollary implies that following limits exist
	\begin{align}\label{limite-Lnf-Ln1}
	\lim_{n\to\infty}
	\frac{\mathscr{L}^n_{A}(f)(-1^{\infty})}
	{\mathscr{L}^n_{A}(1)(-1^{\infty})}
	\quad
	\text{and}
	\quad
	\lim_{n\to\infty}
	\frac{\mathscr{L}^n_{A}(f)(1^{\infty})}
	{\mathscr{L}^n_{A}(1)(1^{\infty})},
	\end{align}
for all increasing function $f$  depending only on a finite number
of coordinates.

Let us consider a very important class of increasing functions.
For any finite set $B\subset \mathbb{N}$ we define
$\varphi_{B}:X\to \mathbb{R}$ by
\begin{align}\label{def-phi-B}
\varphi_{B}(x)= \prod_{i\in B}\frac{1}{2}(1+x_i).
\end{align}
For convenience, when $B=\emptyset$ we define $\varphi_B(x)\equiv 1$.
The function $\varphi_B$ is easily seen to be increasing
since it is finite product of non-negative increasing functions.
For any $i\in\mathbb{N}$ the following holds
$
(1/2)(1+x_i)\frac{1}{2}(1+x_i)
=
(1/4)(1+2x_i+x_i^2)
=
(1/4)(1+2x_i+1)
=
(1/2)(1+x_i).
$
Therefore for any finite subsets $B,C\subset \mathbb{N}$
we have $\varphi_{B}(x)\varphi_{C}(x)=\varphi_{B\cup C}(x)$.
This property implies that the collection $\mathscr{A}$
of all linear combinations of $\varphi_B$'s  is in fact
an algebra of functions
\[
\mathscr{A}
\equiv
\left\{
\sum_{j=1}^{n}a_j\varphi_{B_j}
: n\in\mathbb{N}, a_j\in\mathbb{R} \ \text{and}\
B_j\subset\mathbb{N}\ \text{is finite}
\right\}.
\]
It is easy to see that $\mathscr{A}$
is an algebra of functions that separate points and contains
the constant functions. Of course, $\mathscr{A}\subset C(X)$.
Since $X$ is compact it follows from
the Stone-Weierstrass theorem that $\mathscr{A}$ is dense
in $C(X)$.

Since $\varphi_{B}$ depends only on $\#B$ coordinates follows from
\eqref{limite-Lnf-Ln1} and the linearity of the Rulle operator
that we can define a linear functional $F^{+}:\mathscr{A}\to\mathbb{R}$
by the following expression
\[
F^{+}(\sum_{j=1}^{n}a_j\varphi_{B_j})
=
\sum_{j=1}^{n}a_j 	
	\lim_{n\to\infty}
	\frac{\mathscr{L}^n_{A}(\varphi_{B_j})(1^{\infty})}
	{\mathscr{L}^n_{A}(1)(1^{\infty})}.
\]
From the positivity of the Ruelle operator it follows that
$F^{+}$ is continuous. Indeed,
\begin{align*}
F^{+}(\sum_{j=1}^{n}a_j\varphi_{B_j})
&=
\sum_{j=1}^{n}a_j 	
	\lim_{n\to\infty}
	\frac{\mathscr{L}^n_{A}(\varphi_{B_j})(1^{\infty})}
	{\mathscr{L}^n_{A}(1)(1^{\infty})}
\\
&=
	\lim_{n\to\infty}
	\frac{\mathscr{L}^n_{A}(\sum_{j=1}^{n}a_j \varphi_{B_j})(1^{\infty})}
	{\mathscr{L}^n_{A}(1)(1^{\infty})}
\\
&\leq
	\lim_{n\to\infty}
	\frac{\mathscr{L}^n_{A}
	(\|\sum_{j=1}^{n}a_j \varphi_{B_j}\|_{\infty}\cdot 1)(1^{\infty})}
	{\mathscr{L}^n_{A}(1)(1^{\infty})}
\\
&=
\|\sum_{j=1}^{n}a_j \varphi_{B_j}\|_{\infty}.
\end{align*}
We prove analogous lower bounds and therefore
\[
\Big| F^{+}(\sum_{j=1}^{n}a_j\varphi_{B_j})\Big|
\leq
\|\sum_{j=1}^{n}a_j \varphi_{B_j}\|_{\infty}.
\]
Since $\mathscr{A}$ is dense in $C(X)$ the functional
$F^+$ can be extended to a bounded linear functional defined
over all $C(X)$. Clearly $F^{+}$ is positive bounded functional
and $F^{+}(1)=1$.
Therefore it follows from the Riesz-Markov theorem that
there exists a probability measure $\mu^{+}$ such that
\[
F^{+}(f) = \int_{X} f\, d\mu^{+}.
\]

For the functions $\varphi\in\mathscr{A}$ a bit more can be said
\begin{align}\label{eq-integrais-ruelle}
\lim_{n\to\infty}
\frac{\mathscr{L}^n_{A}(\varphi)(\pm 1^{\infty})}
{\mathscr{L}^n_{A}(1)(\pm 1^{\infty})}
=
F^{\pm}(\varphi)
=
\int_{X} \varphi\, d\mu^{\pm}.
\end{align}
Of course, the probability measures $\mu^{\pm}$ and
both depends on $A$ which in turn depends
on $(a_n)_{n\in\mathbb{N}}$ and $h$, but we are omitting such
dependence to lighten the notation.

\begin{theorem} \label{yyt}
Let $A$ be a potential as in Corollary
\ref{cor-monotonicidade-quociente-ruelles} and
$\mu^{\pm}$ the probability measures defined above. Then
\begin{align}\label{eq-mu-mais-igual-mu-menos}
\mu^{+}=\mu^{-}
\quad \Longleftrightarrow\quad
\int_{X}x_i\, d\mu^{+}(x) = \int_{X}x_i\, d\mu^{-}(x)
\qquad \forall\, i\in\mathbb{N}
\end{align}
\end{theorem}

\begin{proof}
If  $\mu^{+}=\mu^{-}$  then the rhs of \eqref{eq-mu-mais-igual-mu-menos}
is obvious. Conversely, assume that
lhs of \eqref{eq-mu-mais-igual-mu-menos} holds.
Let $\varphi\in\mathscr{A}$ be an increasing function.
From the Corollary \ref{cor-monotonicidade-quociente-ruelles}
and the identity \eqref{eq-integrais-ruelle} we have
\begin{align}\label{eq-aux1-nao-neg-dif-int-varphi}
0
\leq
\lim_{n\to\infty}
\frac{\mathscr{L}^n_{A}(\varphi)(1^{\infty})}
{\mathscr{L}^n_{A}(1)(1^{\infty})}
-
\lim_{n\to\infty}
\frac{\mathscr{L}^n_{A}(\varphi)((-1)^{\infty})}
{\mathscr{L}^n_{A}(1)((-1)^{\infty})}
=
\int_{X}\varphi \, d\mu^{+}-\int_{X}\varphi \, d\mu^{-}.
\end{align}

Fix a finite subset $B\subset \mathbb{N}$ and define
\[
\psi(x) = \sum_{i\in B} x_i -\varphi_{B}(x).
\]
Clearly we have $\psi\in \mathscr{A}$.
We claim that $\psi$ is increasing function.
To prove the claim take $x,y\in X$ such that $y\succeq x$.
If $x_i=y_i$ for all $i\in B$ then $\psi(x)=\psi(y)$
and obviously $\psi(x)\leq \psi(y)$. Suppose that
there exist $j\in B$ such that $-1=x_j<y_j=1$.
Since $\varphi_{B}$ takes only values zero or one, we have
$-1\leq \varphi_{B}(x)-\varphi_{B}(y)\leq 1$,
by definition of $j$ we have $y_j-x_j=2$ so
\begin{align*}
\psi(y) - \psi(x)
&=
\sum_{i\in B} y_i -\varphi_{B}(y)
-
\sum_{i\in B} x_i +\varphi_{B}(x)
\\
&=
\sum_{i\in B} (y_i-x_i) +\varphi_{B}(x)-\varphi_{B}(y)
\\
&=
\sum_{i\in B\setminus\{j\}} (y_i-x_i) +2
+\varphi(x)-\varphi_{B}(y)
\geq
\sum_{i\in B\setminus\{j\}} (y_i-x_i)
\geq 0.
\end{align*}
Since $\psi\in\mathscr{A}$ and increasing
follow from \eqref{eq-aux1-nao-neg-dif-int-varphi}
and the hypothesis that
\begin{align*}
0\leq
\int_{X} \psi\, d\mu^{+}
-
\int_{X} \psi\, d\mu^{-}
&=
\int_{X} [\sum_{i\in B}x_i - \varphi_{B}(x)]\, d\mu^{+}(x)
-
\int_{X} [\sum_{i\in B}x_i - \varphi_{B}(x)]\, d\mu^{-}(x)
\\[0.2cm]
&=
\int_{X} \varphi_{B}(x)\, d\mu^{-}(x)
-\int_{X} \varphi_{B}(x)\, d\mu^{+}(x)
\\
&\leq
0.
\end{align*}
Therefore for any finite $B\subset\mathbb{N}$ we
have
\[
\int_{X} \varphi_{B}(x)\, d\mu^{-}(x)
=
\int_{X} \varphi_{B}(x)\, d\mu^{+}(x).
\]
By linearity of the integral the above
indentity extends to any function $\varphi\in\mathscr{A}$.
Since $\mathscr{A}$ is a dense subset of $C(X)$ it
follows that $\mu^{+}=\mu^{-}$.
\end{proof}

We denote by $\mathcal{G}^{*}(A)$ the set of eigenprobabilities for the dual of the Ruelle operator of the potential $A$.

The set $\mathcal{G}^{\mathrm{DLR}}(A)$ is the set of probabilities satisfying the DLR condition (see \cite{CL-rcontinuas-2016}). DLR probabilities are very much studied on Statistical Mechanics.

\begin{theorem}[See \cite{CL-rcontinuas-2016}]
If $A:X\to \mathbb{R}$ is any continuous potential then
\[
\mathcal{G}^{*}(A)
=\mathcal{G}^{\mathrm{DLR}}(A).
\]
\end{theorem}

\begin{theorem}[Uniqueness]
Let $A$ be a potential as in Corollary \ref{cor-monotonicidade-quociente-ruelles}.
If $\mu^{+}=\mu^{-}$ then $\mathcal{G}^{*}(A)$ is a
singleton.
\end{theorem}

\begin{proof}
Since $A(x)=hx_0+x_0\sum_{n}a_nx_n$ and $\sum_{n}a_n<\infty$ then
$A$ is continuous. For this potential it is very well known that
the set $\mathcal{G}^{\mathrm{DLR}}(A)$ is the closure of the convex
hull of all the cluster points of the sequence $(\mu_n^{y})_{n\in\mathbb{N}}$
for all $y\in X$.

Given a finite subset $B\subset\mathbb{N}$ let
$n\geq 1$ be such that $B\subset \{1,\ldots,n\}$.
From  Corollary \ref{cor-dom-estocastica-mu-finite-vol} we get
\[
\int_{X} \varphi_{B}\, d\mu_{n}^{-}
\leq
\int_{X} \varphi_{B}\, d\mu_{n}^{y}
\leq
\int_{X} \varphi_{B}\, d\mu_{n}^{+}.
\]
If $\mu$ is any cluster point of $(\mu_n^{y})_{n\in\mathbb{N}}$ then
follows from the last inequalities that
\[
\int_{X} \varphi_{B}\, d\mu^{-}
\leq
\int_{X} \varphi_{B}\, d\mu
\leq
\int_{X} \varphi_{B}\, d\mu^{+}.
\]

The above inequality is in fact an equality by
hypothesis. By linearity we can extend
the last conclusion to any function $g\in \mathscr{A}$
and therefore
follows from the denseness of $\mathscr{A}$
and from the hypothesis that
\[
\int_{X} f\, d\mu^{-}
=
\int_{X} f\, d\mu
=
\int_{X} f\, d\mu^{+},
\qquad
\forall \, f\in C(X).
\]
Thus proving that the set of the cluster points of
$(\mu_n^{y})_{n\in\mathbb{N}}$ is a singleton, implying
that $\mathcal{G}^{\mathrm{DLR}}(A)=\mathcal{G}^{*}(A)$
is also a singleton.
\end{proof}

\section{Symmetry Preserving Eigenmeasures and Examples}

In this section we work with the symbolic space
$X =\{a,-a\}^\mathbb{N}$ , where $a>0$ will be convenient
choose latter.

\begin{definition}
We say that a continuous potential $A: X \to \mathbb{R}$
is \textbf{mirrored}\break if
$A(x_0,x_1,x_2,\ldots) = A(-x_0,-x_1,-x_2,\ldots)$,
for all $x\in X$.
We denote by $\mathcal{I}$ the set of mirrored potentials.
\end{definition}

As an example consider Ising type potentials of the form
\[
A(x_0,x_1,x_2,\ldots)
=
x_0 x_1 a_1 + x_0 x_2 a_2 +\ldots + x_0 x_n a_n+\ldots,
\]
where $\sum_{n} |a_n|<\infty$.
Of course, the Dyson potential with $h=0$
is an element on the above family of potentials.

If in addition we assume that in the above potential that
$a_j\geq 0$, for all $j\geq 1$
then we have that $A\in\mathcal{E}$.
In this section we established some results for
potentials of this form but not living in the space $\mathcal{E}$.

\begin{proposition}
If $A\in \mathcal{I}$  and $\varphi$ is an eigenfunction
for $\mathscr{L}_A$ associated to an eigenvalue $\lambda$ of $\mathscr{L}_{A}$,
then $\tilde{\varphi}:X\to \mathbb{R}$, given by
$\tilde{\varphi} (x_0,x_1,x_2,\ldots)\equiv  \varphi(-x_0,-x_1,-x_2,\ldots) $
is also an eigenfunction associated to $\lambda$.
\end{proposition}

\begin{proof}
Indeed, for any $(x_0,x_1,\ldots)\in X$ we have
\[
\lambda
\varphi(x_0,x_1,\ldots)
=
e^{ A(a,x_0,x_1,\ldots) }\, \varphi(a,x_0,x_1,\ldots)+
e^{ A(-a,x_0,x_1,\ldots) }\, \varphi(-a,x_0,x_1,\ldots).
\]
Since $A\in\mathcal{I}$ follows from the last equation that
\begin{multline*}
\lambda
\tilde{\varphi}(-x_0,-x_1,\ldots)
=
e^{ A(-a,-x_0,-x_1,\ldots) }\, \tilde{\varphi}(-a,-x_0,-x_1,\ldots)+
\\
e^{ A(a,-x_0,-x_1,\ldots) }\, \tilde{\varphi}(a,-x_0,-x_1,\ldots).
\end{multline*}
By taking $y_j=-x_j$, for all $j\in \mathbb{N}$, we get
\[
\lambda
\tilde{\varphi}(y_0,y_1,\ldots)
=
e^{ A(-a,y_0,y_1,\ldots) }\, \tilde{\varphi}(-a,y_0,y_1,\ldots)+
e^{ A(a,y_0,y_1,\ldots) }\, \tilde{\varphi}(a,y_0,y_1,\dots),
\]
which means that $\tilde{\varphi}$ is an eigenfunction associated
to the eigenvalue $\lambda$.
\end{proof}

We point out that the above analytical reasoning applies whenever is well defined - even if $\varphi$ is not continuous.

\begin{remark}
If $A\in \mathcal{I}$ and $\varphi$ is a strictly positive and continuous eigenfunction
associated to the spectral radius $\lambda_A$ then
$\varphi(x_0,x_1,\ldots)=  \varphi(-x_0,-x_1,\ldots)$. This equality follows from the above and the
uniqueness of a strictly positive eigenfunction associated to the maximal eigenvalue
for a continuous potential, for details
see \cite{MR1085356}.  Figure \ref{fig5} on the section \ref{sim} illustrate this fact.
\end{remark}

\medskip

\begin{proposition}
Let $A\in \mathcal{I}$  and $\nu$ an eigenprobability  for $\mathscr{L}_{A}^{*}$,
associated to the eigenvalue $\lambda_{A}$. If
$\tilde{\nu}$ is the unique Borel probability measure
defined by the following functional equation
\[
\int_{X} f(x)\, d \tilde{\nu}(x)
=
\int_{X} f( -x_0,-x_1,-x_2,...)\, d\nu(x),
\qquad \forall f\in C(X)
\]
then $\tilde{\nu}$ is also an eigenprobability
associated to the eigenvalue $\lambda_{A}$.
\end{proposition}

\begin{proof}
It is enough to prove that for any real continuous function $g:X\to\mathbb{R}$,
we have
\[
\lambda_{A} \int_{X} g(x) \, d\tilde{\nu}(x)
=
\int_{X}
e^{ A(a,x_0,x_1,\ldots) }\, g(a,x_0,x_1,\ldots)+
e^{ A(-a,x_0,x_1,\ldots) }\, g(-a,x_0,x_1,\ldots)\,
d \tilde{\nu}(x).
\]
Given any continuous function $f(x)=f(x_0,x_1,x_2,...)$ it follows from the hypothesis that
\begin{align*}
\lambda_{A}
\int_{X}& f(x) \, d \nu(x)
=
\int_{X}
e^{ A(a,x_0,x_1,\ldots) }\, f(a,x_0,x_1,\ldots)+
e^{ A(-a,x_0,x_1,\ldots) }\, f(-a,x_0,x_1,\ldots)\,
d \nu(x)
\\[0.2cm]
&=
\int_{X}
e^{ A(-a,-x_0,-x_1,\ldots) }\, f(a,x_0,x_1,\ldots)+
e^{ A(a,-x_0,-x_1,\ldots) }\, f(-a,x_0,x_1,\ldots)\,
d \nu(x)
\\[0.2cm]
&=
\int_{X}
e^{ A(-a,x_0,x_1,\ldots) }\, f(a,-x_0,-x_1,\ldots)+
e^{ A(a,x_0,x_1,\ldots) }\, f(-a,-x_0,-x_1,\ldots)
\, d \tilde{\nu}(x).
\end{align*}
By taking $f(x_0,x_1,\ldots)= g (-x_0,-x_1,\ldots)$
in the above expression, we get
\begin{align*}
\lambda_{A}
\int_{X}&
g(x_0,x_1,\ldots)  \, d\tilde{\nu}(x)
=
\lambda_{A}
\int_{X} g(-x_0,-x_1,\ldots)  \,d \nu(x)
=
\lambda_{A}
\int_{X} f(x) \, d \nu(x)
\\
&=
\int_{X}
e^{ A(-a,x_0,x_1,\ldots) }\, f(a,-x_0,-x_1,\ldots)+
e^{ A(a,x_0,x_1,\ldots) }\, f(-a,-x_0,-x_1,\ldots)\,
d \tilde{\nu}(x)
\\
&=
\int_{X}
e^{ A(-a,x_0,x_1,\ldots) }\, g(-a,x_0,x_1,\ldots)+
e^{ A(a,x_0,x_1,\ldots) }\, g(a,x_0,x_1,\ldots)
\, d \tilde{\nu}(x).
\end{align*}
\end{proof}

\begin{remark}\label{remark-conserva-simetria}
If the eigenprobability $\nu$ associated to $\lambda_{A}$,
of a mirrored potential is unique, then for any continuous function
$f:X\to\mathbb{R}$ we have that
\[
\int_{X} f(x) d \nu (x)
=
\int_{X} f( -x_0,-x_1,-x_2,...) d\nu(x).
\]
\end{remark}

We shall observe  that the results of this section can be applied to
the Dyson potential, under appropriate assumptions and restrictions.

\subsection{The Binary Model} \label{bina}

In this section we take $X =\{-1/2,1/2\}^\mathbb{N}$.
We recall that any point in $\tilde{x}\in [-1,1]$ has a binary expansion
of the form $\tilde{x}= x_0 +   x_1 2^{-1} + x_2  2^{-2}+\ldots$, where
$x_i\in \{-1/2,1/2\}$ for all $i\in\mathbb{N}$.
Using this binary expansion
we get a bijection (with only countably many exceptions)
between the points of the symbolic space $x=(x_0,x_1,x_2,\ldots)\in X$
with the points of the closed unit ball in $\mathbb{R}$, i.e.,
$\tilde{ x}\in [-1,1]\subset \mathbb{R}$.
For example, the point $\tilde{ x}=1$ can be represent
in the symbolic space by $x=(1/2,1/2,1/2,\ldots),$
and similarly the point $\tilde{ x}=-1$ is represented by $x=(-1/2,-1/2,-1/2,\ldots).$
Furthermore, the value $1/2\in [-1,1]$
can be written as $(1/2,-1/2,1/2,1/2,1/2,...)$, or, as  $(1/2,1/2,-1/2,-1/2,-1/2,...)$,
and etc...
Changing $(x_0,x_1,x_2,\ldots)$ to $(-x_0,-x_1,-x_2,\ldots)$ corresponds to change $\tilde{x}$ to $-\tilde{x}$.

An important point on the reasoning below is that  for any $n\geq 2$,
if $x_0=1/2$, then we have
\[
x_0 x_1 + x_0 x_2 2^{-1} +\ldots+ x_0 x_n 2^{-(n-1)}
\geq
x_0 x_{n+1} 2^{-n} + x_0 x_{n+2} 2^{-(n+1)}+\ldots.
\]
In other words, the tail is smaller than
the first $n$ terms of the series.
Therefore a point $\tilde{x}$ represented by
a sequence $x=(1/2,x_2,x_3,\ldots)$ is such that $\tilde{x}\geq 0.$
On the other hand, if it is represented by a sequence like
$x=(-1/2,x_2,x_3,\ldots)$, then  $\tilde{x}\leq 0.$

Note that if $ (y_0,y_1,y_2,\ldots) $ and $ (x_0,x_1,x_2,\ldots) $
are two comparable points in $(X,\succeq)$, we have
$ (y_0,y_1,y_2,..) \succeq (x_0,x_1,x_2,....) $,  if and only if,
$
\tilde{y}= y_0 + y_1 2^{-1} + y_2  2^{-2}+\ldots
\geq
\tilde{x}= x_0 + x_1 2^{-1} + x_2  2^{-2}+\ldots
$.

By using the binary expansion we can think of $f: X \to \mathbb{R}$
as a function $f: [-1,1] \to \mathbb{R}$. By abusing notation we will
write $f(x_0,x_1,x_2,..)=f(\tilde{x}).$
Whenever $f: X\to \mathbb{R}$ is continuous and
take same values where the representation is not unique, then
the associated function $f: [-1,1] \to \mathbb{R}$ is also continuous.
Clearly if $f: X \to \mathbb{R}$ is increasing function and
$(y_0,y_1,y_2,\ldots) \succeq (x_0,x_1,x_2,\ldots)$, then
we have $f(\tilde{y}) \geq f(\tilde{x})$.

Note that the shift on $\{-1/2,1/2\}^\mathbb{N}$ can be represented under such change of coordinates as the expanding transformation $T:[-1,1] \to [-1,1] $
such that $T(x) =2x-1$, for $x>0$, and by $T(x)=2x+1$, when $x\leq 0$. The inverse branches of $T$ are
$ y \to \frac{y+1}{2}$ and  $ y \to \frac{y-1}{2}$.

\begin{example}[The Binary Model]
Let $X =\{-1/2, 1/2\}^\mathbb{N}$  and $A:X\to\mathbb{R}$  the
Ising type potential given by
\[
A(x)
=
A(x_0,x_1,x_2,...)
=
x_0 x_1 + x_0 x_2 2^{-1} + x_0 x_3 2^{-2}+\ldots+x_0 x_n 2^{-n+1} +\ldots
\]
which will be called the \textbf{binary potential}.
Clearly $A$ is a Lipchitz potential and $A \in \mathcal{I}$.
Note that if $x=(x_0,x_1,\ldots)$ is a representation of $\tilde{x}\in [-1,1]$
then $A$ can be represented as a function $A:[-1,1] \to [-1,1]$,
where $A(x)= x-1/2$, for $x>0$, and,  $A(x)= (-1/2)-x$, for $x\leq 0$.
If $\tilde{x}$ is associated to  $(x_0,x_1,x_2,\ldots)$, then
$A(1/2,x_0.x_1,..)= (1/2)\tilde{x}$ and
$A(-1/2,x_0.x_1,..)= -(1/2) \tilde{x}+ 1$.
So the equation for the eigenfunction of the Ruelle operator is
\[
\mathscr{L}_{A} (\varphi) (\tilde{x})
=
e^{\tilde{x}/2} \varphi ((\tilde{x}+1)/2) +
e^{-\tilde{x}/2\,\,} \varphi (  (\tilde{x}-1)/2)
=
\lambda
\varphi(\tilde{x}).
\]

The function $\varphi(\tilde{x}) = \frac{3}{4} \tilde{x}^2 + \frac{3}{32} ( 15 + \sqrt{353})$ is not an eigenfunction of
$\mathscr{L}_{A}$ but the functions  $  \frac{1}{32} ( 49 + \sqrt{353}) \varphi$ and  $\mathscr{L}_{A}(\varphi)$
are quite close on the interval $[-1,1]$ (see Figure \ref{fig1} on
Section \ref{sim}). The corresponding Taylor series  around zero agree up to order two.

From numerical point of view one could get better approximations
- polynomials of higher order -
by using, for example, the command \texttt{Expand} in Mathematica
(and solving some equations in order to get the coefficients of the polynomial).

\end{example}

\section{Involution Kernel Representations of Eigenfunctions} \label{inv}

In this section we obtain a semi-explicit expressions for eigenfunctions
of the Ruelle operator $\mathscr{L}_{A}$, associated to the maximal eigenvalue for a
large class of potentials $A$. The main technique used here is the involution
Kernel. Before present its definition and some of its basic properties we
need set up some notations.

From now on the symbolic space is taken
as $X=\{-1,1\}^\mathbb{N}$ and we use the notation
$\hat{X}\equiv \{-1,1\}^\mathbb{Z}$.
The set of all sequences written in ``backward direction''
$\{(\ldots,y_2,y_1): y_j\in\{-1,1\} \}$ will be denoted by
$X^*$, and given a pair $x\in X$ and $y\in X^{*}$ we defined
$( y|x) \equiv (\ldots,y_2,y_1|x_1,x_2,\ldots)$.
Using such pairs we can identify $\hat{X}$ with the cartesian product
$X^*\times X$. This bi-sequence space is sometimes called the natural extension of $X$.
The left shift mapping on $\hat{X}$ will be denoted by $\hat{\sigma}$
and defined as usual by
\[
\hat{\sigma}(\ldots,y_2,y_1|x_1,x_2,x_3\ldots)
=
(\ldots,y_2,y_1,x_1|x_2,x_3,\ldots).
\]

\begin{definition}
Let $A:X\to \mathbb{R}$ be a continuous  potential
(considered as a function on $\hat{X}$).
We say that a continuous function $W:\hat{X}\to\mathbb{R}$  is an involution kernel for $A$,
if there exists continuous potential $A^*:X^*\to\mathbb{R}$
(considered as a function on $\hat{X}$) such that
for any $a\in \{-1,1\}$, $x\in X$ and
$y\in X^*$, we have
\begin{equation}\label{eqWkernel}
(A^*+W)( ya|x)=(A+W)( y|ax).
\end{equation}
\end{definition}

We say $A^*$ the dual of the potential $A$ (using $W$) and
$A$ is symmetric if for some involution kernel $W$,
we have $A=A^*$.

To simplify the notation we write simply $A(x)$, $A^* (y)$ and $ W(y|x)$
during the computations.
For general properties of involutions kernels, the reader can
see the references  \cite{MR2210682},
\cite{MR3377291} and \cite{MR3471364}.

In several examples one can get the explicit expression for $W$ and $A^*$
(see  section 5 of  \cite{MR3114331}).

The reader should be warned that the involution kernel $W$
is not unique.

\begin{example}

Let $A:X\to \mathbb{R}$ be a continuous potential (considered as a function on $\hat{X}$)
given by $A(x_1,x_2,x_3,...) = a_1\,x_1+a_2\, x_2 +...+ x_n\, a_n+\ldots$,
where $\sum_{k}\sum_{n\geq k} |a_n|<\infty$. A large class of such potentials were carefully studied
in \cite{JLMS:JLMS12031} and spectral properties of the Ruelle
operator were obtained there.

We claim that $A^*=A$ (for some choice of $W$).
Indeed, let $k= \sum_{j\geq 2} a_j.$ and define for any $(x|y)\in \hat{X}$ the
following function
\begin{align*}
W(y|x)
&=
[\,\ldots+ (k- (a_2+a_3+a_4))\, y_4 + (k- (a_2+a_3))\, y_3 +  (k-a_2)\, y_2 + k\, y_1\,
\\
&\phantom{==\ldots}+
k\, x_1+ (k-a_2)\, x_2+ (k- (a_2+a_3))\, x_3+ (k- (a_2+a_3+a_4))\, x_4+\ldots\,].
\end{align*}
Using the hypothesis placed on the coefficients $a_n$'s we can rewrite
\[
W(y|x) =  \sum_{i\geq 1} (x_i+y_i)\, (a_{i+1} + a_{i+2} +\ldots).
\]

A simple computation shows that for
any $a\in \{-1,1\}$, $x\in X$ and
$y\in X^*$, we have the following identity
\[
A(a\,y)+ W(y\,a\,|\,x)
=
(A+W)( y\,a\,|\,x)=(A+W)( y\,|\,a\,x)= A(a\,x) + W(y\,|a\,\,x),
\]
thus showing that $A$ is simetric, i.e., $A=A^*$.

In \cite{JLMS:JLMS12031} is shown that the main eigenfunction of $\mathscr{L}_A$
is given by
\[
\varphi(x)= \exp(\alpha_1 x_1 + \alpha_2 x_2 +\ldots),
\quad\text{where}\ \ \alpha_n = a_{n+1}+a_{n+2}+\ldots
\]
and the main eigenvalue is
$\exp(\sum_{j=1}^\infty\, a_j) + \exp(-\sum_{j=1}^\infty\, a_j).$
If $\beta>0$ is fixed and the coefficients $(a_n)_{n\geq 1}$
are given by  $a_j= \beta \,j^{-\gamma}$ for all $j \in \mathbb{N}$,
we get that the main eigenvalue is equals to $2 \cosh (\beta \zeta(\gamma))$.
Figure \ref{fig8} gives an idea what is the shape of the graph of this eigenfunction.

\end{example}

\begin{example}

For an Ising type potential of the form
$A(x)=  \sum_{j=1}^\infty x_1\,x_j a_{j+1}$,
we can formally written an expression for the involution kernel $W$, which is
\begin{equation}\label{bomge}
W(y|x)=
y_1 \,( \sum_{j=1}^\infty \,x_j\, a_j )\,+
y_2 \,( \sum_{j=1}^\infty \,x_j\, a_{j+1} )\,+ \ldots+
y_k \,( \sum_{j=1}^\infty \,x_j\, a_{j+k-1} )+ \ldots.
\end{equation}
Of course, to give a meaning for the above expression
some restrictions need to be imposed on $(a_n)_{n\geq 1}$.
We return to this issue latter.

One can show that when $a_n= \lambda^n$, for $0< \lambda \leq 0.5$, this involution kernel satisfies
the twist property (see definition on \cite{MR3114331}). One has to consider the lexicographic order on this definition which is a total order.

\end{example}

\medskip

\begin{theorem} Let $A$ be a continuous potential for which there exists an involution
kernel $W$. Let $A^*$ be a continuous potential satisfying the
equation \eqref{eqWkernel} and $\nu_{A^*}$ an eigenprobability of $\mathscr{L}_{A^*}^{*}$,
associated to the spectral radius $\lambda_{A^*}$.
Then the function
\begin{equation} \label{expinv}
\varphi(x)
\equiv
\int_{X^*} e^{W(y|x)} d \nu_{A^*}(y)
\end{equation}

is a continuous positive eigenfunction for the Ruelle operator $\mathscr{L}_{A}$ associated
to $\lambda_{A^*}$.
\end{theorem}
\begin{proof}
Since $\mathscr{L}_{A^*}^* \nu_{A^*} = \lambda_{A^*} \nu_{A^*}$.
we have for any continuous function $f:\hat{X}\to\mathbb{R}$ the following identity
\[
\int_{X^{*}} f(y) d \nu_{A^*}(y)
=
\lambda_{A^*}
\int_{X^{*} } \mathscr{L}_{A^*}(f)(y) \,d \nu_{A^*}(y).
\]
On the other hand,
\begin{align*}
\mathscr{L}_A (\varphi) (x)
&=
\mathscr{L}_A  \left(  \int_{X^{*}} e^{W(y|x)} d \nu_{A^*}(y) \right)(x)
=
\sum_{a=\pm 1} e^{ A(a\, x)} \int_{X^{*}} e^{W(y|ax)} d \nu_{A^*}(y)
\\
&=
\int_{X^{*} } \left[\sum_{a=\pm 1} e^{ A(a\, x)} e^{W(y|ax)}\right] d\nu_{A^*}(y)
=
\int_{X^{*} } \left[ \sum_{a=\pm 1} e^{ A^*(y\, a)} e^{W(ya|x)} \right]d \nu_{A^*}(y)
\\
&=
\int_{X^{*} } \mathscr{L}_{A^*}\Big(e^{W(\cdot|x)}\Big)(y) \, d \nu_{A^*}(y)
=
\lambda_{A^*}
\int_{X^{*}}  e^{W(y|x)} d \nu_{A^*}(y)
\\[0.2cm]
&=
\lambda_{A^*} \varphi(x).
\end{align*}

\end{proof}

Figure \ref{fig10} on the Section \ref{sim} provides a numerical comparison between
the approximations computed by using the involution kernel
and the explicit expression of the eigenfunction for some product type potentials,
see \cite{JLMS:JLMS12031}.

\subsection{Involution Kernel and the Dyson model}

Now we consider some continuous Ising type potentials of the form
\[
A(x)=
\sum_{j=1}^\infty \frac{x_1 x_{j+1}}{j^{\gamma}},
\quad\text{where} \ \gamma>1
\]
and the formal series
\begin{equation}\label{bom}
W(y|x)=
y_1 \,( \sum_{j=1}^\infty \,x_j\, j^{-\gamma} )\,+
y_2 \,( \sum_{j=1}^\infty \,x_j\, (j+1)^{-\gamma} )\,+\ldots+
y_k \,( \sum_{j=1}^\infty \,x_j\, (j+k-1)^{-\gamma} )+\ldots
\end{equation}
Such $W$ is well defined and is continuous, whenever $\gamma>2$.
If the terms in the above formal sum can be rearranged
then we can show that $W(x|y)=W(y|x)$. In such cases, a
simple algebraic computation give us the following relation
\begin{equation}\label{sym}
A(ay) + W(ya\,|x)= (A+W)( ya|x)=(A+W)( y|\,a\,x)= A(a\,x) + W(y\,|\,a\,x)
\end{equation}
for any $a \in \{-1,1\}$, showing that $A$ is symmetric.
By multiplying both sides of the above equation by $\beta>0$
we get that $\beta W$ is an involution kernel for $\beta A$.
\medskip

A natural question: is the involution kernel $\nu_{A}\times\nu_{A}$ almost everywhere
well-defined, where $\nu_A$ is some eigenprobability ?
If the answer is affirmative, then above formula for $W$
provides an measurable eigenfunction.

Let $\widetilde{X}$ be the subset of all $x=(x_1,x_2,\ldots)$ in $X=\{-1,1\}^\mathbb{N}$
such that there exist and $N$ such that $x_j=-x_{j+1}$ for all $N \leq j$.
Note that the set $\widetilde{X}$ is dense subset of $X$ and if
$x \in \widetilde{X}$, then their preimages are also in $\widetilde{X}$.

Suppose $1< \gamma <2$, then, for each $k$ we have that
$\sum_{j=1}^\infty \,x_j\, (j+k-1)^{-\gamma} $ converges
and it is of (at most) order $k^{-\gamma}$, when $k \to \infty$.
In this way for such $x\in \widetilde{X}$ we get that $W(y|x)$ is well defined for all $y$.

\begin{theorem} \label{kkh}
Consider the potential
$A(x)= \sum_{j=1}^\infty j^{-\gamma} x_1x_{j+1}$, where $1<\gamma<2$.
There exist a non-negative function $\widetilde{\varphi}:\widetilde{X} \to \mathbb{R}$
such that for any $x\in \widetilde{X}$ we have
\[
\mathscr{L}_{A} (\widetilde{\varphi})(x)
=
\lambda_{A}\, \widetilde{\varphi} (x) ,
\]
where $\lambda_{A}$ is the spectral radius of $\mathscr{L}_{A}$, acting on $C(X)$.
\end{theorem}
\begin{proof}
Let $\nu_{A}$ the  eingenprobability of $\mathscr{L}^*_{A}$, associated to the
spectral radius $\lambda_{A}$, that is, $\mathscr{L}_{A}^* \nu_{A}= \lambda_{A}\nu_{A}$.
We denote by $\nu_{A^*}$ the  eingenprobability of $\mathscr{L}_{A^*}$.
Since $A=A^*$, we get that $\nu_{A^*}= \nu_{A}$.

Note that for all $x\in \widetilde{X}$ the expression \eqref{bom} for
$W(y|ax)$ is well-defined for any $a$ and $y$.
Finally, for such $x$ we can proceed as in the proof of Theorem \ref{kkh}
to verify that
\[
0
\leq
\widetilde{\varphi} (x)
\equiv
\int_{X^{*}} e^{W(y|x)} d \nu_{A} (y)
\]
is a solution to the eigenvalue problem
$\mathscr{L}_{A} (\widetilde{\varphi})(x) = \lambda_{A} \widetilde{\varphi}(x)$.
\end{proof}

In section \ref{sim} we present some numerical data sketching the shape
of the graph of eigenfunction associated to Dyson potentials
(see Figure \ref{fig5}).

\subsection{Topological Pressure of some Long-Range Ising Models}

Now we show how to use the involution kernel representation
of the main eigenfunction, to obtain bounds on the main eigenvalue $\lambda_{\beta A}$,
where $A$ is an Ising type potential, of the form
$A(x)= \sum_{j=1}^\infty j^{-\gamma} x_1x_{j+1}$, where $\gamma>2$ and $\beta>0$.
For such potentials the main eigenfunction $\varphi_{\beta}$, associated to the main eigenvalue $\lambda_\beta$,
is given by $\varphi_{\beta}(x) = \int_{X^*} e^{\beta \,W(y|x)} d \nu_{\beta A} (y)$ and
positive everywhere. From the definitions we have
\begin{align*}
\lambda_\beta  \varphi_{\beta}  (1,1,1,..)
&=
e^{\beta A(1,1,1,..)   } \varphi_{\beta} (1,1,...) +
e^{ \beta A(-1,1,1,..) } \varphi_{\beta} (-1,1,...)
\end{align*}
and therefore
\begin{align*}
\lambda_\beta
&=
e^{\beta A(1,1,1,..)   } +
e^{ \beta A(-1,1,1,..) } \frac{\varphi_{\beta} (-1,1,...)}{\varphi_{\beta}  (1,1,1,..)}
\\[0.3cm]
&=
e^{\beta \zeta(\gamma) } +
e^{ -\beta\zeta(\gamma)}
\frac{\varphi_{\beta} (-1,1,...)}{\varphi_{\beta}  (1,1,1,..)}
\\[0.3cm]
&=
e^{\beta \zeta(\gamma) } +
e^{ -\beta\zeta(\gamma)}
\frac{\int_{X^*}\exp(\beta W(y|-11^{\infty})) d\nu_{\beta A}(y)}
{\int_{X^*}\exp(\beta W(y|1^{\infty})) d\nu_{\beta A}(y)}.
\end{align*}
Let us foccus on the integrals appearing above.
By using the expression \eqref{bom} we have
\begin{align*}
\frac{
\int_{X^*}
\exp\left(
-y_1 \beta ( \zeta(\gamma)-1)\,+
\beta\sum_{n\geq 2}y_n \,(-n^{-\gamma}+ \sum_{j\geq 2} (j+n-1)^{-\gamma} )
\right) d\nu_{\beta A}(y)
}
{
\int_{X^*}
\exp\left(
y_1 \beta \zeta(\gamma)+
\beta \sum_{n\geq 2}y_n \,( n^{-\gamma}+ \sum_{j\geq 2} (j+n-1)^{-\gamma} )
\right) d\nu_{\beta A}(y)
}
\end{align*}
Note that the numerator is a product of an decreasing function by an
increasing and the denominator is a product of two increasing functions
therefore we can use the FKG inequality
to get an upper bound to the above fraction which is given by
\begin{align*}
\frac{\int_{X^*}
e^{-y_1 \beta ( \zeta(\gamma)-1)}
d\nu_{\beta A}(y)}
{\int_{X^*}
e^{y_1 \beta \zeta(\gamma)}
d\nu_{\beta A}(y)
}
\
\frac{
\int_{X^*}
\exp\left(
\beta\sum_{n\geq 2}y_n \,(-n^{-\gamma}+ \sum_{j\geq 2} (j+n-1)^{-\gamma} )
\right)
d\nu_{\beta A}(y)
}
{
\int_{X^{*}}
\exp\left(
\beta \sum_{n\geq 2}y_n \,( n^{-\gamma}+ \sum_{j\geq 2} (j+n-1)^{-\gamma} )
\right)
d\nu_{\beta A}(y)
}
\end{align*}
The first quotient above can be explicit computed as follows
\begin{align*}
\frac{\int_{X^*}
e^{-y_1 \beta ( \zeta(\gamma)-1)}
d\nu_{\beta A}(y)}
{\int_{X^*}
e^{y_1 \beta \zeta(\gamma)}
d\nu_{\beta A}(y)
}
&=
\frac{
	e^{\beta ( \zeta(\gamma)-1)}
	\nu_{\beta A}(y_1=-1)
	+
	e^{-\beta ( \zeta(\gamma)-1)}
	\nu_{\beta A}(y_1=1)
}
{
	e^{\beta \zeta(\gamma)}
	\nu_{\beta A}(y_1=1)
	+
	e^{-\beta \zeta(\gamma)}
	\nu_{\beta A}(y_1=-1)
}
\\[0.3cm]
&=
\frac{\cosh(\beta (\zeta(\gamma)-1)) }{\cosh(\beta\zeta(\gamma))}
<
1,
\end{align*}
where we have used that $\mathcal{G}^*(A)$ is a singleton and
Remark \ref{remark-conserva-simetria} to conclude that
the probabilities $\nu_{\beta A}(y_1=\pm 1)=1/2$.
The second quotient can be bounded by one using again the FKG inequality
and Remark \ref{remark-conserva-simetria}.
The above estimates implies that the pressure of this long range Ising model
is bounded by
\[
P(\beta A)=
\log \lambda_{\beta A}
<
\log (2\cosh(\beta\zeta(\gamma)) ).
\]
Tights lower bounds are much harder to obtain.
Anyway this computation let it clear the relevance of the involution
kernel to obtain upper bounds for the pressure functional.

\section{$\pmb{L^2(X,\nu_{A})}$ Eigenfunctions for the Ruelle Operator} \label{weakCL}

The space $C(X)$ is not a suitable space to solve the main eigenvalue
problem for the Ruelle Operator $\mathscr{L}_{A}$ for a general
continuous potential $A$. In \cite{JLMS:JLMS12031} and
\cite{CL-rcontinuas-2016} the authors
exhibit a family of potentials for which the Ruelle operator has no
continuous main eigenfunction. For example, if the potential $A$ is
of the form
\[
A(x) = \sum_{n\geq 1} \frac{x_n}{n^{\alpha}},
\quad \text{where}\  \ 1<\alpha<2
\]
the authors prove the existence of a measurable non-negative ``eigenfunction''
taking values from zero to infinity in any fixed cylinder set of $X$.
Of course, such function can not be extended to a continuous function
defined in whole space $X$.  Some extension to Lebesgue space should be
considered in order these functions can be viewed as legitimate eigenfunctions.
In this section we study the main eigenvalue problem for the
natural extension of the Ruelle operator to $L^2(X,\nu_{A})$.

We start by proving that the Ruelle operator can be extended to
a positive operator defined on $L^2(X,\nu_{A})$
(short notation $L^2(\nu_A)$)
for any continuous
potential $A$.

\begin{lemma}\label{lema-extensao-ruelle-L2}
	Let $A:X\to\mathbb{R}$ be a continuous potential and
	$\nu_A$ any element in $\mathcal{G}^{*}(A)$. Then
	the Ruelle operator $\mathscr{L}_{A}:C(X)\to C(X)$
	can be extended to a bounded linear operator defined
	on $L^2(\nu_{A})$.
\end{lemma}

\begin{proof}
	It is enough to prove that the Ruelle operator is bounded
	on a dense subspace with respect to the $L^2(\nu_A)$-norm.
	Since $X$ is a compact metric space, we have that $C(X)$ is
	a dense subset of $L^2(\nu_A)$.
	Let $\psi$ be a continuous function.
	By using the positivity and duality relation
	of the Ruelle operator we get
	\begin{align*}
	\|\mathscr{L}_{A}(\psi)\|_{L^2(\nu_A)}^2
	&=
	\int_{X} \mathscr{L}_{A}(\psi)\mathscr{L}_{A}(\psi)\, d\nu_{A}
	\leq
	\int_{X} \mathscr{L}_{A}(|\psi|)\mathscr{L}_{A}(|\psi|)\, d\nu_{A}
	\\
	&=
	\int_{X} \mathscr{L}_{A}(\mathscr{L}_{A}(|\psi|)\circ\sigma\cdot |\psi|)\, d\nu_{A}
	=
	\lambda_{A}
	\int_{X} \mathscr{L}_{A}(|\psi|)\circ\sigma\cdot |\psi|\, d\nu_{A}.
	\end{align*}
	Developing the integrand by using the definition of the Ruelle operator
	and the continuity of the potential $A$ we get
	\begin{align*}
	\mathscr{L}_{A}(|\psi|)\circ\sigma\cdot |\psi|
	&=
	|\psi(1x_2x_3\ldots)|\cdot\exp(A(1x_2x_3\ldots))\cdot|\psi(x_1x_2x_3\ldots)|
	\\
	&
	\qquad+
	|\psi(-1x_2x_3\ldots)|\cdot\exp(A(-1x_2x_3\ldots))\cdot|\psi(x_1x_2x_3\ldots)|
	\\[0.3cm]
	&
	\leq
	\exp(\|A\|_{\infty})
	\Big[
	|\psi(1x_2x_3\ldots)|\cdot|\psi(x_1x_2x_3\ldots)|
	\\
	&\qquad +
	\exp(\|A\|_{\infty})
	|\psi(-1x_2x_3\ldots)|\cdot|\psi(x_1x_2x_3\ldots)|
	\Big].
	\end{align*}
	By using this upper bound and the Cauchy-Schwarz Inequality
	we can conclude from the above inequalities that
	\begin{align*}
	\|\mathscr{L}_{A}(\psi)\|_{L^2(\nu_A)}^2
	&\leq
	\lambda_{A}\exp(\|A\|_{\infty})
	\left(\int_{X} \psi^2(1x_2\ldots)\, d\nu_A\right)^{\frac{1}{2}}
	\|\psi\|_{L^{2}(\nu_{A})}
	\\
	&\qquad+
	\lambda_{A}\exp(\|A\|_{\infty})
	\left(\int_{X} \psi^2(-1x_2\ldots)\, d\nu_A\right)^{\frac{1}{2}}
	\|\psi\|_{L^{2}(\nu_{A})}
	\\
	&\leq
	\lambda_{A}\exp(\|A\|_{\infty})
	\left(\int_{X} 1_{\{x_1=1\}}\psi^2\, d\nu_A\right)^{\frac{1}{2}}
	\|\psi\|_{L^{2}(\nu_{A})}
	\\
	&\qquad+
	\lambda_{A}\exp(\|A\|_{\infty})
	\left(\int_{X} 1_{\{x_1=-1\}}\psi^2\, d\nu_A\right)^{\frac{1}{2}}
	\|\psi\|_{L^{2}(\nu_{A})}
	\\[0.3cm]
	&\leq
	2\lambda_{A}\exp(\|A\|_{\infty})
	\|\psi\|_{L^{2}(\nu_{A})}^2.
	\end{align*}
	Thus proving that the Ruelle operator can be extended in
	$L^2(\nu_A)$ to a bounded linear operator.
\end{proof}

Given a continuous potential $A$, a point $z_0\in \Omega$ and $n \in \mathbb{N}$,
is natural to approximate $A$ by a potential $A_n$ defined by the mapping
$ x= (x_1,x_2,..,x_n,x_{n+1},..) \longmapsto A(x_1,x_2,..,x_n,z_0)$.
Note that $A_n$ depends on  a finite number of coordinates and therefore belongs
to the H\"older class.
Typical choices of $z_0$ could be either $1^\infty$ or $(-1)^\infty$.

\begin{lemma}\label{lema-construcao-muA}
	Let $A$ be a continuous potential and
	for each $n\in\mathbb{N}$ we define $A_n(x)=A(x_1,\ldots,x_n,1,1,\ldots)$.
	Let $\varphi_n$ denotes the main eigenfunction of $\mathscr{L}_{A_n}$
	normalized so that $\|\varphi_n\|_{L^1(\nu_n)}=1$, where $\nu_n$
	is the unique eigenprobability of $\mathscr{L}_{A_n}^{*}$.
	Then there exist a $\sigma$-invariant Borel probability measure $\mu_{A}$
	(called asymptotic equilibrium state) such that, up to subsequence,
	\[
	\lim_{n\to\infty} \int_{X} \mathscr{L}_{A_n}(\varphi_n)\psi\, d\nu_{n}
	=
	\lambda_{A}\int_{X} \psi\, d\mu_{A},
	\quad \forall \psi\in C(X)
	\]
\end{lemma}

\begin{proof}
	Let $\lambda_{A_n}$ be the main eigenvalue of the Ruelle operator associated to the potential
	$A_n$ and $\varphi_n$ its normalized eigenfunction, i.e., $\|\varphi_n\|_{L^1(\nu_n)}=1$ .
	By the Corollary 2 of \cite{CL-rcontinuas-2016} we get that $\lambda_{A_n}\to \lambda_{A}$,
	when $n\to\infty$. Since $\varphi_n\geq 0$ and $\|\varphi_n\|_{L^1(\nu_{n})}=1$
	the measure $\mu_n =\varphi_n\nu_n$ is a
	probability measure for each $n\in\mathbb{N}$. Since $X$ is compact,
	there is a probability measure $\mu_A$ such that, up to subsequence,
	$\mu_n\rightharpoonup \mu_A$. Therefore for all $\psi\in C(X)$
	we have
	\begin{align*}
	\int_{X} \mathscr{L}_{A_n}(\varphi_n)\psi\, d\nu_{n}
	&=
	\lambda_n\int_{X} \varphi_n \psi\, d\nu_{n}
	\\
	&=
	\lambda_n\int_{X} \psi\, d\mu_{n}
	\xrightarrow{n\to+\infty}
	\lambda_{A}\int_{X} \psi\, d\mu_{A}.
	\qedhere
	\end{align*}
\end{proof}

\begin{theorem}[Equilibrium States]\label{teo-estados-equilibrio}
	Let $A:X\to\mathbb{R}$ be a continuous potential.
	Suppose that $\{A_n:n\in\mathbb{N}  \}$ is a sequence
	of H\"older potentials such that $\|A_n-A\|_{\infty}\to 0$,
	when $n\to\infty$. Then any the asymptotic equilibrium state $\mu_{A}$
	given by Lemma \ref{lema-construcao-muA} is an equilibrium state for $A$.
\end{theorem}

\begin{proof}
	See \cite{CL-rcontinuas-2016}.
\end{proof}

\begin{definition}[Weak-Solution]\label{weak-sol-eigenvalue-prob}
	Let $\nu_{A}$ be an eigenprobability for $\mathscr{L}^*_{A}$
	and $\mu_{A}$ be given by the Lemma \ref{lema-construcao-muA}.
	We say that a non-negative function $\varphi_{A}\in L^2(\nu_{A})$
	is a weak solution to the eigenvalue problem for the Ruelle
	operator, if $\|\varphi_{A}\|_{L^1(\nu_{A})}=1$ and
	\[
	\int_{X} \mathscr{L}_{A}(\varphi_{A})\psi\, d\nu_{A}
	=
	\lambda_{A}\int_{X} \psi\, d\mu_{A},
	\]
	for all $\psi\in C(X)$.
\end{definition}

\medskip

\begin{proposition}
	If $A$ is a H\"older potential and $\varphi_{A}$ is the
	main eigenfunction of $\mathscr{L}_{A}$, then $\varphi_{A}$
	is a weak solution to the eigenvalue problem in the sense
	of the Definition \ref{weak-sol-eigenvalue-prob}.
\end{proposition}

\begin{proof}
	We first consider a sequence of potentials $A_n:X\to\mathbb{R}$
	defined, for each $n\in\mathbb{N}$ and $x\in X$, by
	$A_n(x)=A(x_1,\ldots,x_n,1,1,\ldots)$. Note that $A$ is H\"older
	and one can immediately check that
	$\|A_n-A\|_{\infty}\to 0$, when $n\to\infty$.
	As in the Lemma \ref{lema-construcao-muA},
	let $\varphi_n$ denotes the main eigenfunction of $\mathscr{L}_{A_n}$
	normalized so that $\|\varphi_n\|_{L^1(\nu_n)}=1$, where $\nu_n$
	is the unique eigenprobability of $\mathscr{L}_{A_n}^{*}$.
	Then there exist a $\sigma$-invariant Borel probability measure $\mu_{A}$
	such that, up to subsequence,
	\[
	\lim_{n\to\infty} \int_{X} \mathscr{L}_{A_n}(\varphi_n)\psi\, d\nu_{n}
	=
	\lambda_{A}\int_{X} \psi\, d\mu_{A},
	\quad \forall \psi\in C(X)
	\]
	From the Theorem \ref{teo-estados-equilibrio} we have that  $\mu_A$
	is an equilibrium state for $A$. Since the potential $A$ is H\"older
	its unique equilibrium state is known to be given by the probability measure
	$\gamma_{A}=\varphi_{A}\nu_{A}$ and therefore $\mu_A=\gamma_{A}$.
	This last equality together with the hypothesis give us
	for all $\psi\in C(X)$ that
	\begin{align*}
	\int_{X} \mathscr{L}_{A}(\varphi_{A})\psi\, d\nu_{A}
	=
	\int_{X} \lambda_{A}\varphi_{A}\psi\, d\nu_{A}
	&=
	\lambda_{A}\int_{X} \psi\, d[\varphi_{A}\nu_{A}]
	\\
	&=
	\lambda_{A}\int_{X} \psi\, d\gamma_{A}
	\\
	&=
	\lambda_{A}\int_{X} \psi\, d\mu_{A}.
	\end{align*}
\end{proof}

The main tool in this section is the
Lions-Lax-Milgram Theorem and it is used to provide weak solutions to
the eigenvalue problem for the Ruelle operator, see \cite{MR1422252}
for a detailed proof of this result.

\begin{theorem}[Lions-Lax-Milgram]
	Let $H$ be a Hilbert space and $V$ a normed space,
	$B:H\times V\rightarrow \mathbb{R}$
	so that for each $v\in B$ the mapping
	$h\mapsto B(h,v)$ is continuous. The following are equivalent:
	for some constant $c>0$,
	$$
	\inf_{\|v\|_{V}=1}\sup_{\|h\|_H\leq 1}|B(h,v)|\geq c;
	$$
	for each continuous linear functional $F\in V^\ast$,
	there exists $h\in H$ such that
	$$
	B(h,v)=F(v)\quad \forall v\in V.
	$$
\end{theorem}

\begin{theorem}\label{teo-ruelle-lax-milgram}
	Let $A:X\to\mathbb{R}$ be a continuous potential,
	$\nu_{A}$ be an element of $\mathcal{G}^*(A)$
	and $\mu_A$ as constructed in Lemma \ref{lema-construcao-muA}.
	Assume that there is $K>0$ such that
	for all $v\in C(X)$ we have  $\|v\|_{L^2(\nu_A)}\leq K\|v\|_{L^2(\mu_A)}$.
	Then there exist a weak solution $\varphi_{A}\in L^2(\nu_A)$
	to the eigenvalue problem for the Ruelle operator.

We get the same claim using an alternative hypothesis:
\[
	\limsup_{n\to\infty} \int_{X} \varphi_n^2 \, d\nu_n \ \ <+\infty,
\]
where $\varphi_n$ and $\nu_{n}$ are chosen as in Lemma \ref{lema-construcao-muA} (see Remark \ref{re4})

\end{theorem}

\begin{proof}

We will prove the theorem assuming that: there is $K>0$ such that
	for all $v\in C(X)$ we have  $\|v\|_{L^2(\nu_A)}\leq K\|v\|_{L^2(\mu_A)}$.

	The main idea of the proof is to use the Lions-Lax-Milgram Theorem
	with the space $H=L^2(\nu_A)$, $V=(C(X),\|\cdot\|_{L^2(\nu_{A})})$,
	$B:L^2(\nu_A)\times C(X)\to\mathbb{R}$
	and $F:C(X)\to\mathbb{R}$
	given respectively, by
	\[
	B(h,v)  = \int_{X} \mathscr{L}_A(h)v \, d\nu_{A}
	\qquad
	\text{and}
	\qquad
	F(v) = \int_{X} v\, d\mu_{A}.
	\]
	
	In the sequel we prove the coercivity condition of the Lions-Lax-Milgram
	theorem and then the continuity of the bilinear form $B$.
	For any $v\in V$ such that $\|v\|_{L^2(\nu_{A})}=1$ we have
	\begin{align*}
	\int_{X} v^2\, d\nu_{A}
	=
	\frac{1}{\lambda_{A}}\int_{X} \mathscr{L}_{A}(v^2)\, d\nu_{A}
	&=
	\frac{1}{\lambda_{A}}\int_{X} (v^2\circ\sigma)\mathscr{L}_{A}(1)\, d\nu_{A}
	\\[0.3cm]
	&\geq
	\frac{\exp(-\|A\|_{\infty})}{\lambda_{A}}\int_{X} (v^2\circ\sigma)\, d\nu_{A}
	\end{align*}
	and therefore
	\[
	\frac{1}{\| v\circ\sigma \|_{L^2(\nu_{A}) }}
	\geq
	\frac{\exp(-\|A\|_{\infty})}{\lambda_{A}}.
	\]
	Similarly we prove that for all $h\in L^2(\nu_{A})$ we have
	$h\circ\sigma\in L^2(\nu_{A})$.
	So it follows from the elementary properties of the Ruelle operator that
	\begin{align*}
	\sup_{\|h\|\leq 1} \int_{X} \mathscr{L}_{A}(h)v \, d\nu_{A}
	&\geq
	\int_{X} \mathscr{L}_{A}\left(\frac{v\circ\sigma}{\|v\circ\sigma\|_{L^2(\nu_A)}} \right)v \, d\nu_{A}
	\\
	&=
	\frac{1}{\|v\circ\sigma\|_{L^2(\nu_A)}}
	\int_{X} \mathscr{L}_{A}(1)v^2 \, d\nu_{A}
	\\[0.3cm]
	&\geq
	\frac{\exp(-2\|A\|_{\infty})}{\lambda_{A}}
	\int_{X} v^2 \, d\nu_{A}.
	\end{align*}
	From the last inequality we get
	\[
	\inf_{\|v\|_{V}=1}\sup_{\|h\|_H\leq 1}|B(h,v)|
	=
	\inf_{\|v\|_{L^2(\nu_{A})}=1}\sup_{\|h\|\leq 1} \int_{X} \mathscr{L}_{A}(h)v \, d\nu_{A}
	\geq
	\frac{\exp(-2\|A\|_{\infty})}{\lambda_{A}}
	\]
	which proves the coercivity hypothesis.
	
	Now we prove the continuity of the mapping $h\longmapsto B(h,v)$,
	where $v\in L^2(\mu_A)$ is fixed. From Lemma \ref{lema-extensao-ruelle-L2}
	we have that $\mathscr{L}_{A}(h) \in L^2(\nu_A)$
	for every $h\in L^2(\nu_A)$.
	So we can use Cauchy-Schwarz inequality to bound $B(h,v)$
	as follows
	\begin{align*}
	|B(h,v)|
	&=
	\left| \int_{X} \mathscr{L}_{A}(h)v \, d\nu_{A} \right|
	\\
	&\leq
	\left( \int_{X} [\mathscr{L}_{A}(h)]^2\, d\nu_{A} \right)^{\frac{1}{2} }
	\left( \int_{X} v^2\, d\nu_{A} \right)^{\frac{1}{2} }
	\\
	&\leq
	(2\lambda_{A}\exp(\|A\|_{\infty}))^{\frac{1}{2}}
	\|h\|_{L^{2}(\nu_{A})}\cdot  \|v\|_{L^{2}(\nu_{A})},
	\end{align*}
	where the last inequality comes the Lemma \ref{lema-extensao-ruelle-L2}
	proof's. The above inequality proves that $h \longmapsto B(h,v)$ is continuous.
	
	The hypothesis $\|v\|_{L^2(\nu_A)}\leq K\|v\|_{L^2(\mu_A)}$ guarantees the continuity
	of the functional $F$ so we can apply the Lions-Lax-Milgram theorem
	to ensure the existence of a function $\overline{\varphi_{A}}\in L^2(\nu_A)$
	so that
	\begin{align}\label{eq-aux-identidade-varphi-barra-A}
	\int_{X}\mathscr{L}_{A}(\overline{\varphi_{A}})v\, d\nu_A
	=
	\int_{X}v\, d\mu_{A}
	\qquad \forall \ v\in L^2(\mu_A).
	\end{align}
	
	By using the identity \eqref{eq-aux-identidade-varphi-barra-A}
	with $v\equiv 1$ we get
	\begin{align*}
	1
	=
	\int_{X}\mathscr{L}_{A}(\overline{\varphi_{A}}) d\nu_A
	=
	\int_{X}\overline{\varphi_{A}}\ d\big[\mathscr{L}_{A}^{*}\nu_A \big]
	\leq
	\lambda_{A} \|\overline{\varphi_{A}}\|_{L^1(\nu_{A})}.
	\end{align*}
	Therefore the following function
	\[
	\varphi =\frac{\overline{\varphi_{A}}}{\lambda_{A} \|\overline{\varphi_{A}}\|_{L^1(\nu_{A})}}
	\]
	is a non-trivial weak solution for the eigenvalue problem.
\end{proof}

\begin{remark} \label{re4}
	We can weaken the hypothesis $\|v\|_{L^2(\nu_A)}\leq K\|v\|_{L^2(\mu_A)}$ for all
	$v\in C(X)$ of Theorem \ref{teo-ruelle-lax-milgram} by only requiring that
	\[
	\limsup_{n\to\infty} \int_{X} \varphi_n^2 \, d\nu_n \ \ <+\infty,
	\]
	where $\varphi_n$ and $\nu_{n}$ are chosen as in Lemma \ref{lema-construcao-muA}.
	Indeed, in order to get the inequality $\|v\|_{L^2(\nu_A)}\leq K\|v\|_{L^2(\mu_A)}$ from the
	above condition it is enough to note that $\nu_n \rightharpoonup\nu_{A}$ and then
	to apply the Cauchy-Schwarz inequality (we leave the details of the proof to the reader).
\end{remark}

\section{Monotonic Eigenfunctions and Uniqueness} \label{classF}

In this section we follow closely \cite{MR1101084,MR1803460}
adapting, to our context, their results for $g$-measures
to non-normalized potentials.

\begin{definition}[Class $\mathcal{F}$]
We say that a potential $A$ belongs to the class $\mathcal{F}$
if for all $y \succeq  x$ we have both inequalities
$e^{A(-1 x)}+ e^{A(1 x)}\leq e^{A(-1 y)}+ e^{A(1 y)} $,
and
$e^{A (1y)} - e^{A (1x)}\geq 0$.
\end{definition}

Note that the above condition is equivalent to
requiring $\mathscr{L}_A (1)(x)$ and
$\mathscr{L}_A (1_{[1]})(x)$ be increasing functions.
A simple example of a potential belonging to the class $\mathcal{F}$
is given by
$A(x)= a_1 x_1 + a_2 x_2 + a_3 x_3 + a_4 x_4+\ldots+ a_n x_n+\ldots$,
where $a_n \geq 0$.

\begin{proposition}\label{kfe}
If $A\in \mathcal{F}$ and $f$ is an increasing non-negative function,
then $\mathscr{L}_{A}(f)$ is increasing function.
\end{proposition}
\begin{proof}
Suppose $A\in \mathcal{F}$, $f \geq 0$, and $f \in \mathcal{I}$.
Then follows directly from the definition of the class $\mathcal{F} $
that if $y\succeq x$ then
\[
e^{A(-1 x)}- e^{A(-1y)}\leq e^{A(1 y)}- e^{A(1 x)}
\qquad \text{and}\qquad
e^{A (1y)} - e^{A (1x)}\geq 0.
\]
By using the above observations and the definition of the Ruelle operator
we get for $y\succeq x$ the following inequalities
\begin{align*}
\mathscr{L}_{A}(f)(y)-
\mathscr{L}_{A}(f)(x)
&=
e^{A (1y)} f(1y) +e^{A (-1y)}   f(-1y) -  e^{A (1x)} f(1x) - e^{A (-1x)} f(-1x)
\\[0.3cm]
&=
e^{A (1y)} (f(1y) - f(1x) )+  e^{A (-1y)}  (f(-1y) - f(-1x) )
\\
&\qquad +
f(1x)(e^{A (1y)} - e^{A (1x)}) - f(-1x)  (-e^{A (-1y)} + e^{A (-1x)})
\\[0.3cm]
&\geq
e^{A (1y)} (f(1y) - f(1x) )+    e^{A (-1y)}  (f(-1y) - f(-1x) )
\\
&\qquad +
f(1x)(e^{A (1y)} - e^{A (1x)}) - f(-1x)  (e^{A (1y)} - e^{A (1x)})
\\[0.3cm]
&=
e^{A (1y)} (f(1y) - f(1x) )+    e^{A (-1y)}  (f(-1y) - f(-1x) )
\\
&\qquad +
(f(1x)- f(-1x) )(e^{A (1y)} - e^{A (1x)})
\\[0.3cm]
&\geq
0.
\end{align*}
\end{proof}

\begin{corollary}
If $A\in\mathcal{F}$ then for any $n\geq 1$, when defined, the function
\[
x\longmapsto \frac{\mathscr{L}_{A}^n(1)(x)}{\lambda_{A}^n}
\]
is an increasing function.
\end{corollary}

\begin{proof}
If $f:X\to\mathbb{R}$ is a non-negative increasing function and $A\in \mathcal{F}$,
then follows from the previous corollary that
$g(x) \equiv \mathscr{L}_A(f)(x)$ is increasing and from positivity of
the Ruelle operator we get $g \geq 0$.
Therefore we can ensure that $\mathscr{L}_A^2(f)(x)\geq 0$
and increasing. Finally, by a formal induction we get that
$\mathscr{L}_A^n(f)(x)\geq 0$ is monotone for each $n$.
Since $\lambda_{A}^n>0$ and $f\equiv 1$
is non-negative increasing function the corollary follows.
\end{proof}

\begin{remark}
If $A\in \mathcal{F}$ and $A$ is a H\"older potential
then we know that
\[
\frac{ \mathscr{L}_A^n(1)(x) }{\lambda_{A}^n}
\xrightarrow{\ n\to\infty\ }
\varphi(x),
\]
uniformly in $x$, and $\varphi$ is the main eigenfunction of $\mathscr{L}_A$,
associated to $\lambda_{A}$.  Since pointwise limit of increasing functions is an increasing
function it follows that the eigenfunction $\varphi$ is an increasing function.
\end{remark}

We now consider a more general situation than the one
in previous remark.
We assume again  that $A\in\mathcal{F}$, but now we also assume
that the sequence of functions $(\varphi_n)_{n\geq 1}$ given by
\[
\varphi_n(x)\equiv
\frac{1}{n}
\sum_{j=0}^{n-1} \lambda_{A}^{-j}\,
\mathscr{L}_{A}^j(1)(x)
\]
has a pointwise everywhere convergent subsequence $(\varphi_{n_k})_{k\geq 1}$.
We also need to assume that
\[
0<
\liminf_{n\to\infty}\frac{\mathscr{L}_{A}^n(1)(1^\infty)}{\lambda_{A}^{n}}
\leq
\limsup_{n\to\infty}\frac{\mathscr{L}_{A}^n(1)(1^\infty)}{\lambda_{A}^{n}}
<
+\infty.
\]
From monotonicity it will follow that $\lambda_{A}^{-n}\,
\mathscr{L}_{A}^n(1)(x)$ is uniformly bounded away from
zero and infinity in $n$ and $x$.

Under such hypothesis it is simple to conclude that $0\leq \varphi$
is a $L^1(\nu_{A})$ eigenfunction of
$\mathscr{L}_{A}$,
associated to its main eigenvalue $\lambda_A$. If $\varphi((-1)^\infty\geq c>0$, then $0<c\leq \varphi$.

Since the set of all cylinders of $X$ is countable,
up to a Cantor diagonal procedure, we can assume that
the following limits exist for any cylinder set $C$
\begin{align}\label{mu-mais-menos}
\mu^+(C)
\equiv
\lim_{k \to \infty} \frac{1}{n_k}
\sum_{j=0}^{n_k-1}
\frac{\mathscr{L}_A^j( 1_{C})(1^\infty) }{\lambda_{A}^j}
\quad \text{and}\quad
\mu^{-}(C)
\equiv
\lim_{k \to \infty}
\frac{1}{n_k}
\sum_{j=0}^{n_k-1}
\frac{\mathscr{L}_A^j(1_{C})(-1^\infty)}{\lambda_{A}^j}.
\end{align}
By standards arguments one can show that $\mu^{\pm}$ can
be both extended to positive measures on the borelians of $X$
(they are not necessarily probability measures).
These measures satisfy for any continuous function $f$ the following identity
\[
\int_{X} f\,  d\mu^{\pm}
=
\lim_{k \to \infty}
\frac{1}{n_k}
\sum_{j=0}^{n_k-1}
\frac{\mathscr{L}_A^j(f)(\pm 1^\infty) }{\lambda_{A}^{j}}.
\]

We claim that $\mu^{\pm}$ are eigenmeasures associated to $\lambda_{A}$.
Indeed, they are both non-trivial measures since $0<\varphi(-1^{\infty})=\mu^{-}(X)\leq \mu^{+}(X)$
and also bounded measures since $\mu^{-}(X)\leq \mu^{+}(X)\leq \varphi(1^{\infty})<+\infty$.
For any $f\in C(X)$ the condition $\varphi(1^{\infty})<+\infty$ implies
\[
\limsup_{n\geq 1}
\left[
\frac{1}{n} \frac{\mathscr{L}_{A}^{n}(f)(1^{\infty})}{\lambda_{A}^n}
\right]
=0.
\]
From the above observations and the
definition of the dual of the Ruelle operator,
for any continuous function $f$ we have
\begin{align*}
\frac{1}{\lambda_{A}}
\int_{X} f\, d[\mathscr{L}_{A}^*\mu^{+}]
&=
\frac{1}{\lambda_{A}}
\int_{X} \mathscr{L}_{A}(f)\, d\mu^{+}
=
\frac{1}{\lambda_{A}}
\lim_{k \to \infty}
\frac{1}{n_k}
\sum_{j=0}^{n_k-1}
\frac{\mathscr{L}_A^j(\mathscr{L}_{A}(f))(1^\infty) }{\lambda_{A}^{j}}
\\[0.3cm]
&=
\lim_{k \to \infty}
\frac{1}{n_k}
\left[
\frac{\mathscr{L}_A^{n_k}(f)(1^\infty) }{\lambda_{A}^{n_{k}}}
-f(1^{\infty})
+
\sum_{j=0}^{n_k-1}
\frac{\mathscr{L}_A^j(f)(1^\infty) }{\lambda_{A}^{j}}
\right]
\\[0.3cm]
&=
\lim_{k \to \infty}
\frac{1}{n_k}
\sum_{j=0}^{n_k-1}
\frac{\mathscr{L}_A^j(f)(1^\infty) }{\lambda_{A}^{j}}
=
\int_{X} f\, d\mu^{+}.
\end{align*}
The above equation shows that $\mu^{+}$ is an eigenmeasure.
A similar argument applies to $\mu^{-}$ and therefore the claim
is proved.

\begin{proposition}\label{trum}
Let $A$ be a continuous potential, and $\lambda_{A}$
the spectral radius of $\mathscr{L}_{A}$ acting on $C(X)$.
Assume that for any continuous function $f:X\to\mathbb{R}$,
the following limit exists and is independent of $x$
\begin{equation} \label{oiy}
\lim_{n \to \infty}
\frac{1}{n}
\sum_{j=0}^{n-1}
\frac{\mathscr{L}_A^j(f)(x) }{\lambda_{A}^j}
=c(f)
\quad\text{and}\quad
\sup_{n\geq 1}
\left\|
\frac{1}{n}
\sum_{j=0}^{n-1}
\frac{\mathscr{L}_A^j(f)}{\lambda_{A}^j}
\right\|_{\infty}
<
+\infty
\end{equation}
Then $\mathcal{G}^*(A)$ is a singleton.
\end{proposition}

\begin{proof}
Since $A$ is continuous follows from \cite{CL-rcontinuas-2016} that
$\mathcal{G}^*(A)$ is not empty. If $\nu\in \mathcal{G}^*(A)$,
then follows from the basic properties of $\mathscr{L}_A$ that for any $f\in C(X)$
and $j\in\mathbb{N}$ we have
\[
\int_{X} f\,d\nu
=
\int_{X}
\frac{\mathscr{L}_A^j (f)}{\lambda_{A}^j} \, d \nu.
\]
From this identity and the Lebesgue Dominated Convergence Theorem
we have
\[
\int_{X} f\,d\nu
=
\lim_{n \to \infty}
\frac{1}{n}
\sum_{j=0}^{n-1}
\int_{X} \frac{\mathscr{L}_A^j(f)(x)}{\lambda^j} \,
d \nu(x)
=
c(f).
\]
Since the above equality is independent of the choice of $\nu$,
we conclude that $\mathcal{G}^*(A)$ has to be a singleton.
\end{proof}

\begin{proposition} \label{port}
Let $A$ be a potential $\mathcal{F}$ and $\lambda_{A}$ the spectral
radius of $\mathscr{L}_{A}$ acting on the space $C(X)$.
If  for some $x\in X$ we have $0<\inf\{\lambda_{A}^{-n}\mathscr{L}_{A}^n(1)(x): n\geq 1\}$
and for all $B\subset \mathbb{N}$ we have
\begin{equation} \label{por}
\lim_{n \to \infty}
\left|
\frac{1}{n}
\sum_{j=0}^{n-1} \lambda_{A}^{-j}  \mathscr{L}_A^j(\varphi_{B})(1^\infty)
-
\frac{1}{n}
\sum_{j=0}^{n-1} \lambda_{A}^{-j}  \mathscr{L}_A^j(\varphi_{B})(-1^\infty)
\right|
=0.
\end{equation}
Then, there exists a continuous positive eigenfunction
$h$ for the Ruelle operator $\mathscr{L}_A$, associated to $\lambda_{A}$.
Moreover, the measures $\mu^+$ and $\mu^{-}$ defined as in \eqref{mu-mais-menos}
are the same and $\mathcal{G}^*(A)$ is a singleton.
\end{proposition}
\begin{proof}
The first step is to show
that the sequence $(\varphi_n)_{n\geq 1}$ defined by
\[
\varphi_n
=
\frac{1}{n} \sum_{j=0}^{n-1} \lambda_{A}^{-j}\, \mathscr{L}_A^j(1)
\]
has a cluster point in $C(X)$.
The idea is to use the monotonicity of $\varphi_n$ to
prove that this sequence is uniformly bounded and equicontinuous.
In fact, for any $j\geq 1$ we have
\[
\lambda_{A}^{-j} \mathscr{L}_A^j(1)(-1^\infty)
\leq
\int_{X} \lambda_{A}^{-j}\mathscr{L}_A^j(1)\, d\nu_{A}
=
\int_{X} 1\, d\nu_A
=1.
\]
Therefore $\varphi_n (-1^\infty)$ is a bounded sequence of real numbers.
From the hypothesis (\ref{por}), with $B=\emptyset$,
follows that $\varphi_n (1^\infty)$ is also bounded.
Since $\varphi_n$ are increasing function we have the following uniform bound
$|\varphi_n(x)|=\varphi_n(x) \leq \sup_{n\geq 1} \varphi_n(1^\infty) $,
thus proving that $(\varphi_n)_{n\geq 1}$ is uniformly bounded sequence in $C(X)$.
To verify that $(\varphi_n)_{n\geq 1}$ is an equicontinuous family it is enough to
use the following upper and lower bounds
\[
\varphi_n (-1^\infty) - \varphi_n (1 ^\infty)
\leq
\varphi_n(x)- \varphi_n(y)
\leq
\varphi_n (1^\infty) - \varphi_n (-1^\infty), \qquad \forall \ x,y\in X
\]
together with the hypothesis \eqref{por}.
Now the existence of a cluster point for the sequence $(\varphi_n)_{n\geq 1}$
is a consequence of Arzela-Ascoli's Theorem, that is, there is some
$\varphi\in C(X)$ such that $\|\varphi_{n_k}-\varphi\|_{\infty}\to 0$, when $k\to\infty$.
Since $0<\inf\{\lambda_{A}^{-n}\mathscr{L}_{A}^n(1)(x): n\geq 1\}$ follows from
the monotonicity of $\varphi$ that $\varphi(x)\neq 0$.  By using the continuity of $\varphi$
and the argument presented in \cite{MR1085356} to prove uniqueness of the eigenfunctions
one can see that $\varphi(y)\neq 0$ for every $y\in X$.
As we observed next to Remark 4,  $\varphi$ is an
eigenfunction of $\mathscr{L}_{A}$, associated to $\lambda_A$.

Now we will prove the statement about $\mathcal{G}^{*}(A)$.
Since $\varphi_{B}$ is an increasing function follows
from Proposition \ref{kfe} that
\begin{align}\label{phi-B-plus-minus-limits}
\frac{1}{n} \sum_{j=0}^{n-1}\lambda_{A}^{-j} \mathscr{L}_A^j(\varphi_B)(-1^\infty)
\leq
\frac{1}{n} \sum_{j=0}^{n-1}\lambda_{A}^{-j} \mathscr{L}_A^j(\varphi_B)(x)
\leq
\frac{1}{n} \sum_{j=0}^{n-1}\lambda_{A}^{-j} \mathscr{L}_A^j(\varphi_B)(1^\infty)
\end{align}
for all $n\geq 1$. From the above inequality and the hypothesis \eqref{por}
is clear that the limit, when $n\to\infty$,
of the second sum in \eqref{phi-B-plus-minus-limits} exist and is
independent of $x$. Therefore the linear mapping
\[
\mathscr{A}\ni
f\longmapsto
\lim_{n\to\infty}
\frac{1}{n} \sum_{j=0}^{n-1}\lambda_{A}^{-j} \mathscr{L}_A^j(f)(x)
=c(f),
\]
defines a positive bounded operator over the algebra $\mathscr{A}$.
By using the denseness of $\mathscr{A}$ in $C(X)$ and the
Stone-Weierstrass theorem it follows that the above limit
is well-defined and independent of $x$ for any continuous function $f$.
Since $\varphi(1^{\infty})<+\infty$ all the hypothesis of
Proposition \ref{trum} are satisfied and so we can
ensure that $\mathcal{G}^*(A)$ is a singleton,
finishing the proof.
\end{proof}


\bigskip

\section{Numerical Data} \label{sim}

Given a general continuous potential $A$ defined on the symbolic pace $\{-a,a\}^\mathbb{N}$
it would be helpful to get an idea of what one would expect for the corresponding
main eigenfunction and eigenvalue (if they exist).
In this section we present some numerical data, obtained by using suitable approximations
of the previous examples, using the software Mathematica.
The aim is to present numerical data related to approximations of the
eigenfunctions and how the shape of these approximated functions looks like.
Some of the numerical computations are based on rigorous mathematical results - for instance if the potential is
of Holder class - and
the approximations we show give a more concrete idea of the behavior of an eigenfunction.
For some complicated models the numerical data is not backed up by rigorous results and they
can be viewed only as illustration -  some of the observations we made in the paper agree with the numerical
computations.

To plot the graphs in this section  we fixed $a=1/2$ and
used the identification of the points in $X = \{-1/2,1/2\}^\mathbb{N}$,
with their ``binary'' expansion on the interval $[-1,1]$ as described on Section \ref{bina}.
Therefore, a graph of a real function defined on $X$
will be plotted as a graph of a real function function defined on $[-1,1]$.

For example, the Figure \ref{fig1} shows how close numerically are
the Taylor approximation of the function $\varphi$  of Section \ref{bina},
and the graph of its image by the Ruelle operator
\begin{figure}[h!]
    \centering
   \includegraphics[scale=0.4,angle=0]{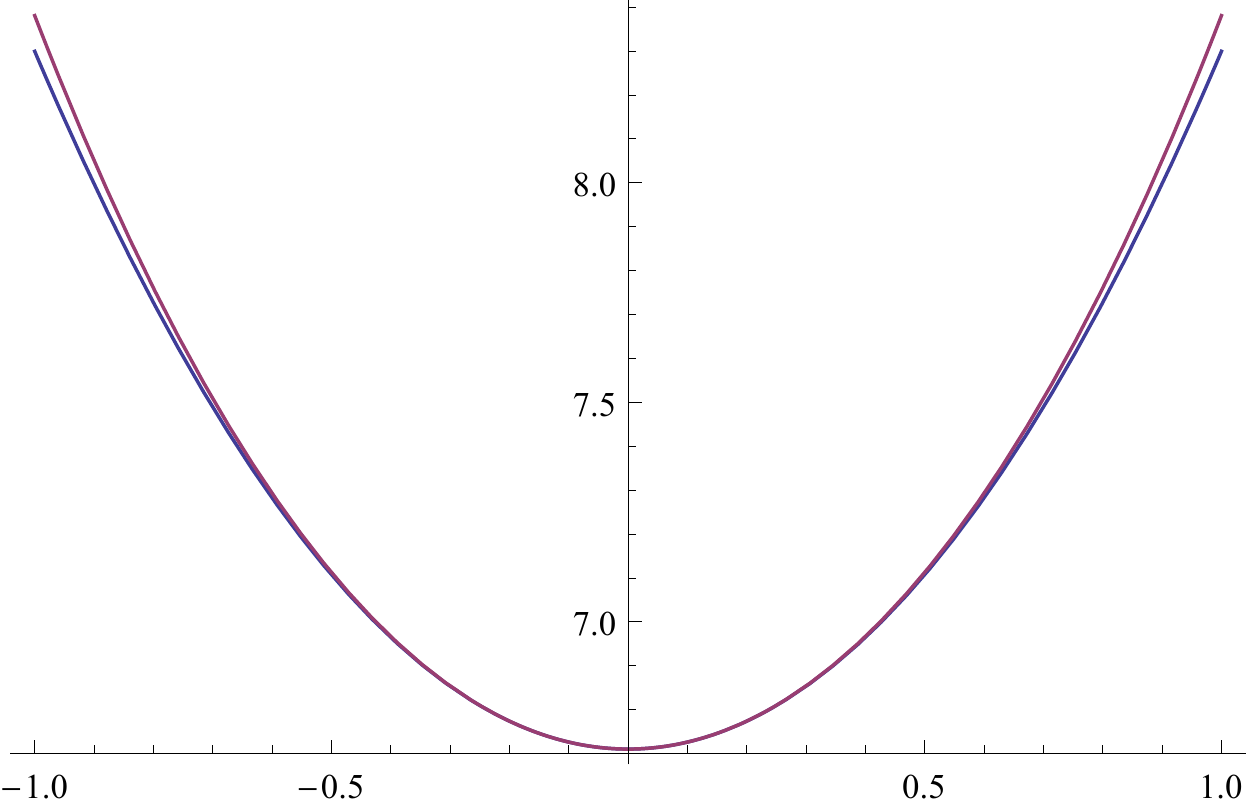}
   \caption{The graphs of  $  \frac{1}{32} ( 49 + \sqrt{353}) \varphi$ (in blue) and  $\mathscr{L}_{A}(\varphi)$ (in red).}
   \label{fig1}
\end{figure}

As we will be interested on potentials defined on $\{-1,1\}^\mathbb{N}$ (as the Dyson model for example) we use
the natural identification of $-1/2$ with $-1$ and  of $1/2$ with $1$.
Under this convention, we present in the sequel the graph of the potential $A$
of the Dyson model for some parameters (see Figures \ref{fig2} and \ref{fig3}).
In all of our numerical computations we worked with potentials approximated by its
16 first terms.

\begin{figure}%
    \centering
    \subfloat{{\includegraphics[width=6cm]{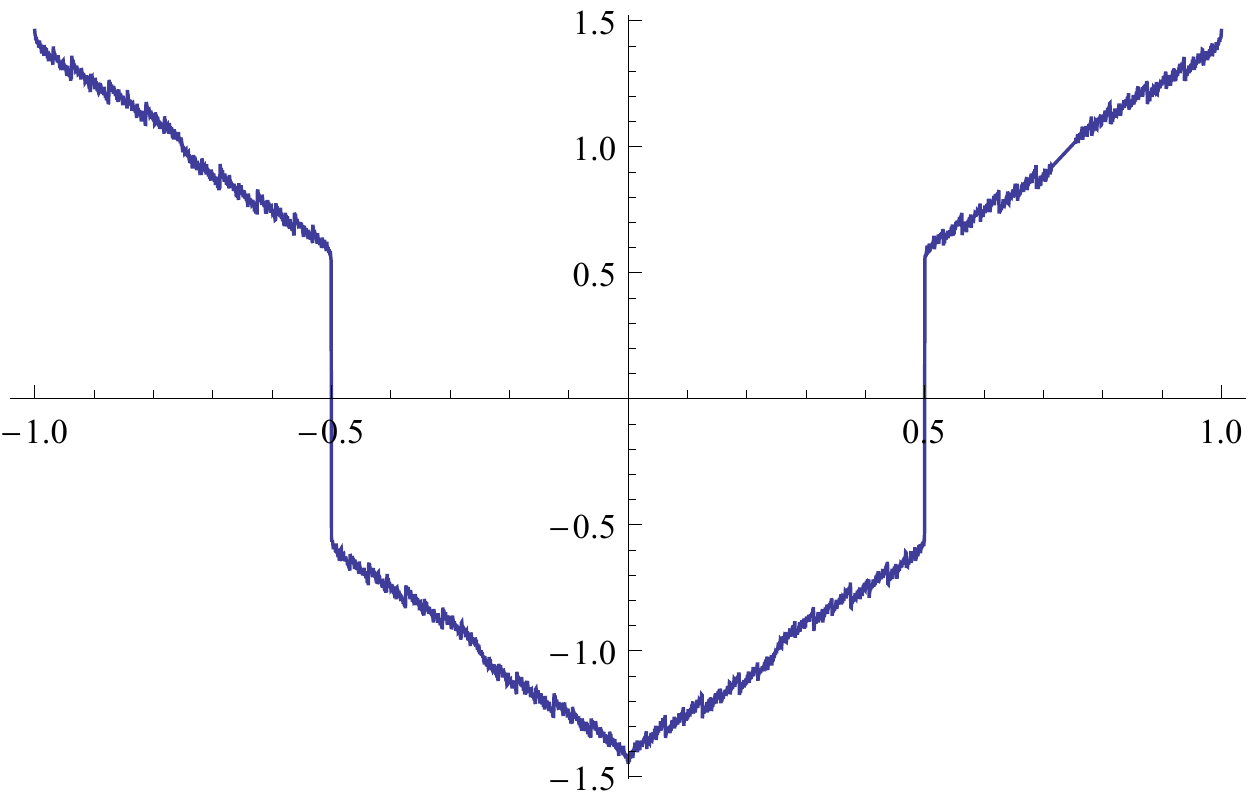} }}%
    \qquad
    \subfloat{{\includegraphics[width=6cm]{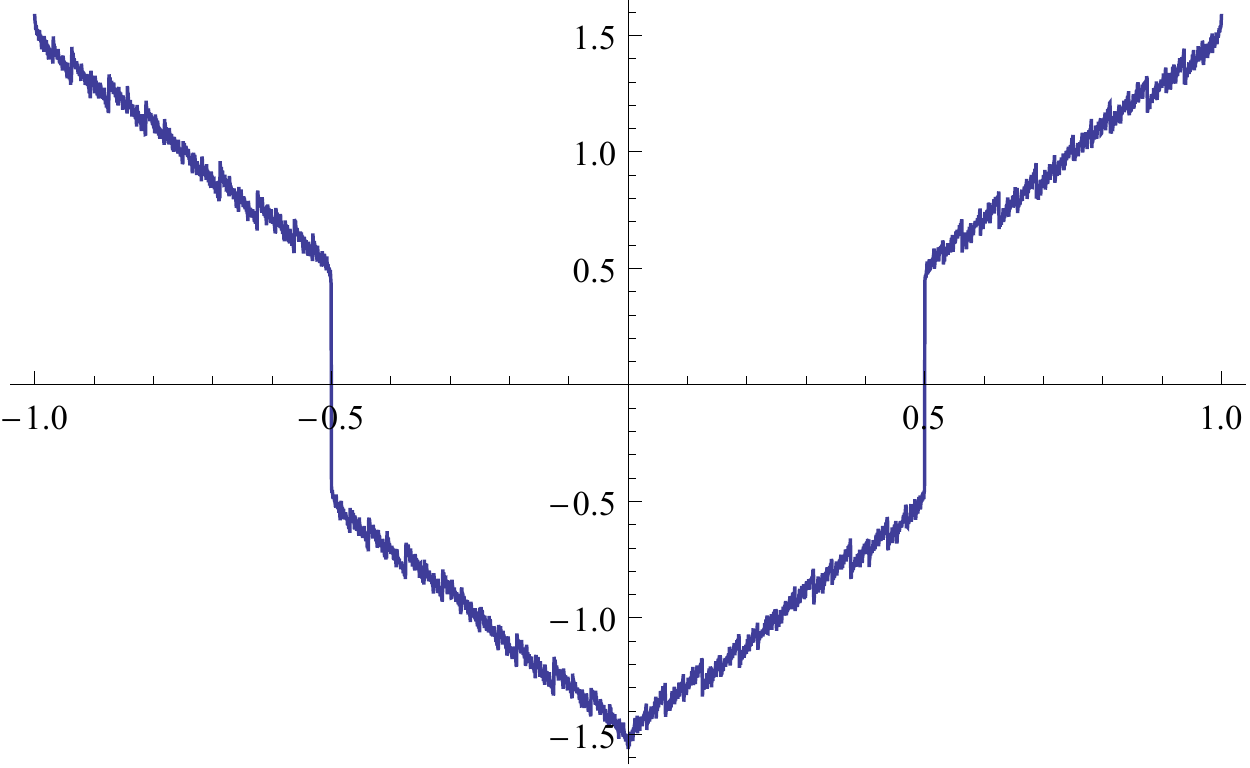} }}%
    \caption{The Dyson potential - (Left) the graph of $A$ for $\gamma=2.2$. (Right) the graph of $A$ for $\gamma=2.0$.}%
    \label{fig2}%
\end{figure}

\begin{figure}%
    \centering
    \subfloat{{\includegraphics[width=6cm]{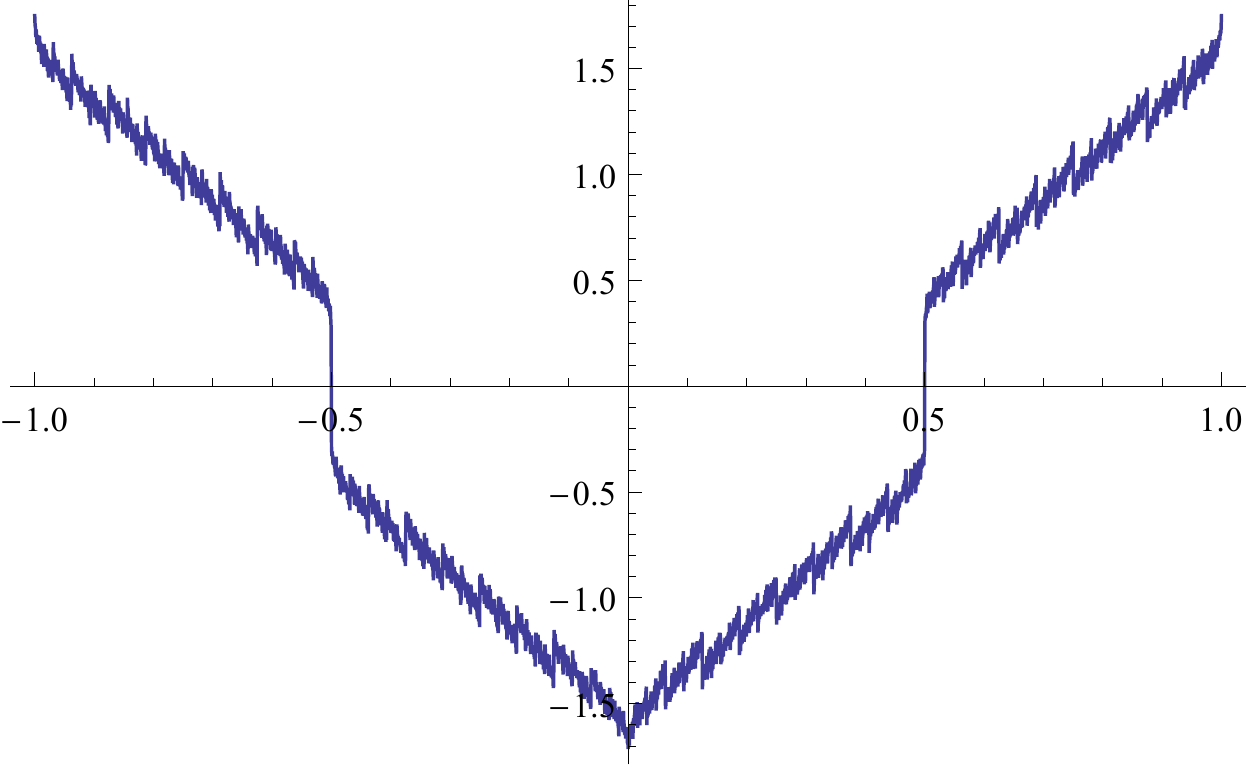} }}%
    \qquad
    \subfloat{{\includegraphics[width=6cm]{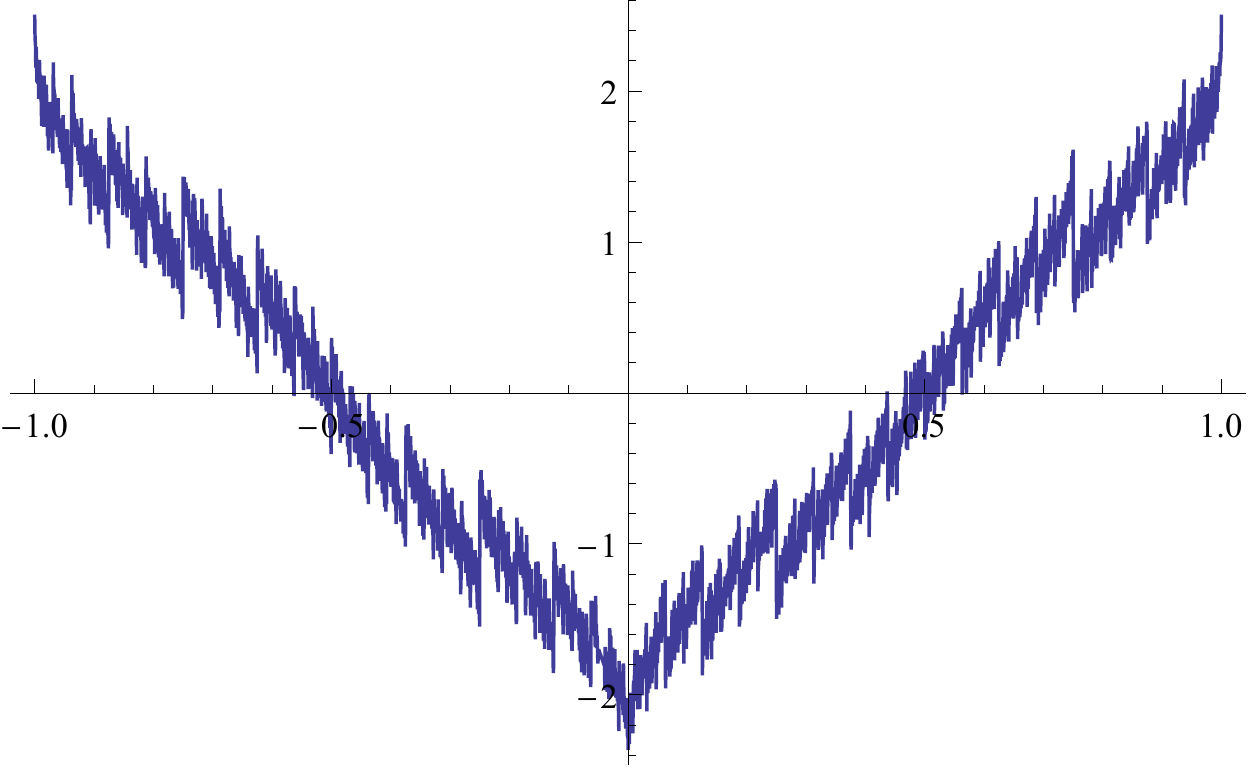} }}%
    \caption{The Dyson potential - (Left) the graph of $A$ for $\gamma=1.88$.
    (Right) the graph of $A$ for $\gamma=1.3$.}%
    \label{fig3}%
\end{figure}

As an approximation of the eigenfunctions we used in all examples of this section the following expression
\[ z_n(x)=\frac{\mathscr{L}_A^n (1) \, (x)} {\mathscr{L}_A^n (1)(1^{\infty})},\]
for $n=3,4,\ldots, 7$.
The reader should have in mind that $z_n(x)$ can only be regarded as a true approximation
as long as the pressure functional can be approximately at a fast rate.
When $\gamma>2$ this is indeed the case. We point out that higher order iterates
(more time consuming for the computer by using a larger $n$) does not change very much the pictures we got.
If the eigenfunction (we want to get) is continuous on $x=1^{\infty}$, then,
this choice of $\mathscr{L}_A^n(1)(1^{\infty})$ seems plausible. However, if the eigenfunction is not defined on
$x=1^{\infty}$ (could be just a measurable function and then we can not be sure)
someone could argue that this choice is unjustifiable. Anyway,
the simulations indicate a stationary pattern, revealing that a careful numerical
analysis is worth to be done.

In order to illustrate the fact that our numerical data are not completely misleading and works
well for H\"older potentials depending on infinite number of variables,
we present the data for the case where the potential is given by
\[
A(x) = x_1 + x_2 2^{-1} + x_3 2^{-2} + \ldots +x_n 2^{-n+ 1}+\ldots
\]
In this case the explicit expression of the eigenfunction $\varphi$
is known and given by the following expression $\varphi(x)= \exp(\alpha_1 x_1 + \alpha_2 x_2 +\ldots)$,
where  $\alpha_n =\sum_{j=n+1}^\infty\, 2^{-j +1} $ (see \cite{JLMS:JLMS12031}).
For this potential the Figure \ref{fig4} illustrate that seven iterates of the Ruelle operator
are enough to get a reasonable approximation of the eigenfunction.
\begin{figure}[h!]
	\centering
	\includegraphics[scale=1,angle=0]{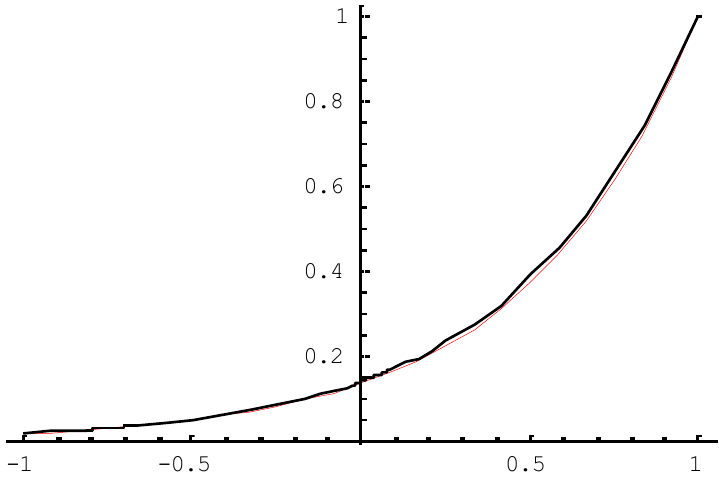}
	\caption{The graphs of the eigenfunction $\varphi$ (red) and $z_7(x)$ (black).}
	\label{fig4}
\end{figure}

Figure \ref{fig5} compare the graph of the Dyson potential
with the parameters $\gamma$ chosen above and below $\gamma=2$
and their respective ``approximated'' eigenfunctions.
The numerical data suggests that the sequence of functions
$x\to z_n(x)$ converges, when $n$ gets large.
There is also an indication that in some cases the better one can hope
is a measurable eigenfunction - but not a continuous one.

\bigskip

\begin{figure}[h!]
	\centering
	\subfloat{\includegraphics[scale=0.6,angle=0]{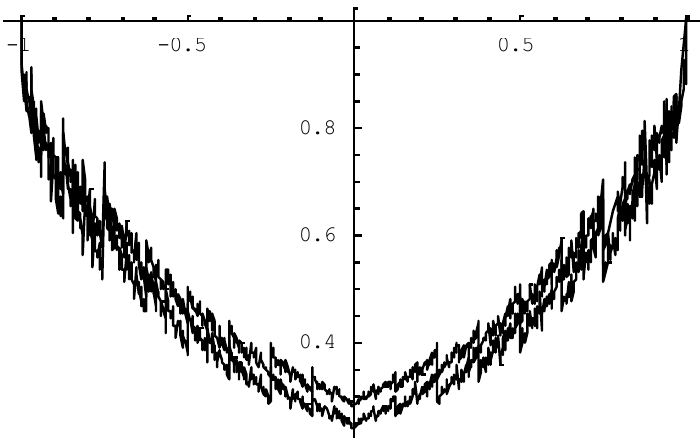}}%
	\qquad\qquad\qquad
	\subfloat{\includegraphics[scale=0.6,angle=0]{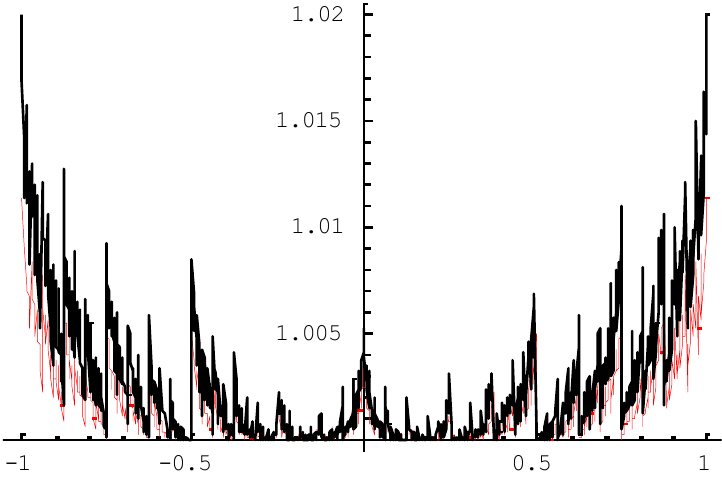}}%
	 \caption{Dyson Potential - (Left) the graphs of $z_3(x)$ and $z_4(x)$ for $\gamma=2.2$.
	 (Right) the graphs of $z_5(x)$ (black) and $z_6(x)$ (red) for $\gamma=1.88$}
	\label{fig5}%
\end{figure}

From now on (Figures \ref{fig7}, \ref{fig9}, \ref{fig8} \ref{fig10}) we present the numerical data for the potential $A(x)=
\sum_{j=1}^\infty j^{-\gamma}x_j$, which was considered
in \cite{JLMS:JLMS12031,MR2218769,MR1101084}.

\begin{figure}[h!]
    \centering
   \includegraphics[scale=0.761,angle=0]{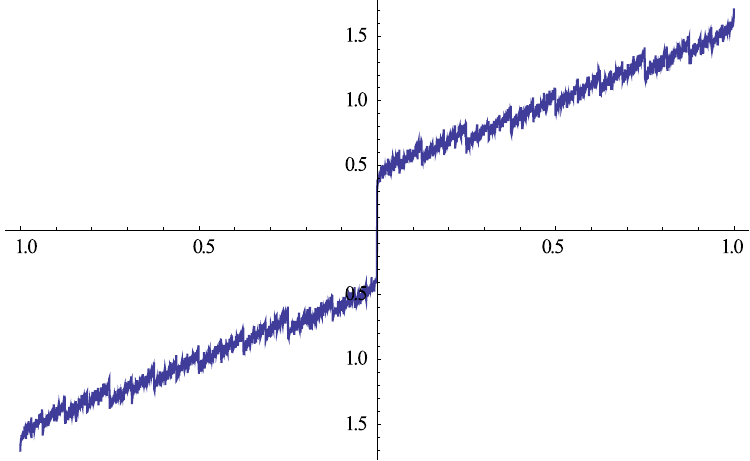}
   \caption{The graph of the potential $A$ when $\gamma=1.88$.
   	The potential $A:\{-1,1\}^\mathbb{N} \to \mathbb{R}$ is continuous and the apparent discontinuity on $x=0$ is
   	due to the multiplicity of the ``binary'' expansion of $0$.
   	Here we using the approximation $\sum_{j=1}^{37}j^{-\gamma}x_j$.}
 \label{fig7}
\end{figure}

When $\gamma>2$ the eigenfunction $\varphi$
is continuous, unique (up to scalar factor) and is given by
the following expression
$\varphi(x)= \exp(\alpha_1 x_1 + \alpha_2 x_2 +\ldots)$,
where  $\alpha_n = \sum_{j=n+1}^\infty\, j^{-\gamma}$.
In the case $\gamma<2$, the expressions are similar
but the analysis of this model is more complex. In such cases
it is also possible to ensure that there exists a measurable eigenfunction $g$
which is defined on the support of the maximal entropy probability.
For almost every $x$ with respect to the maximal entropy probability,
this eigenfunction is given by
$g(x)= \exp(\alpha_1 x_1 + \alpha_2 x_2 +\ldots)$,
where  $\alpha_n = \sum_{j=n+1}^\infty\, j^{-\gamma}.$
There is also another measurable eigenfunction, denoted here by $f$,
which is defined almost everywhere with respect to the eigenprobability for $A$.

\begin{figure}[h!]
    \centering
   \includegraphics[scale=0.7,angle=0]{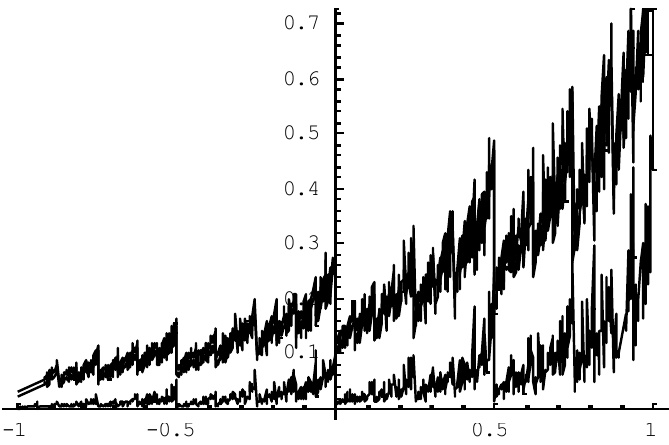}
   \caption{The graph of  $z_6(x)$ (above the graph of $g$) and $g(x)$ for $\gamma=1.88$.}
   \label{fig9}
\end{figure}

\newpage

In this case our numerical data suggest that $z_n(x)$, for large $n$,
will be closed to an eigenfunction which is not the function $g$ (see Figure \ref{fig9}).
In \cite{JLMS:JLMS12031} it is proved for this cases the
existence of more than one measurable eigenfunctions.

\begin{figure}[h!]
	\centering
	\subfloat{{\includegraphics[width=5.0cm]{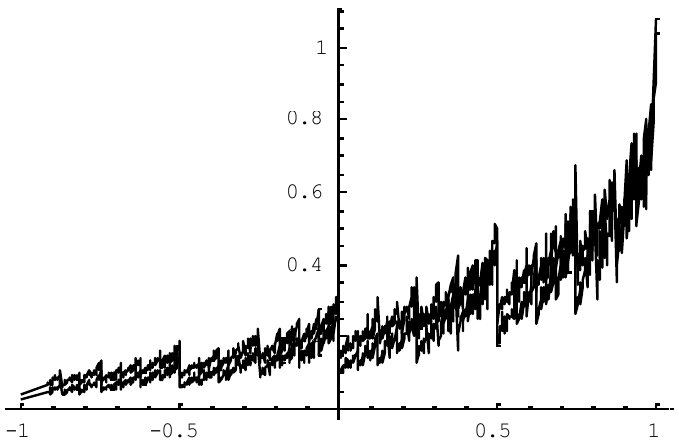} }}%
	\qquad
	\subfloat{{\includegraphics[width=5.0cm]{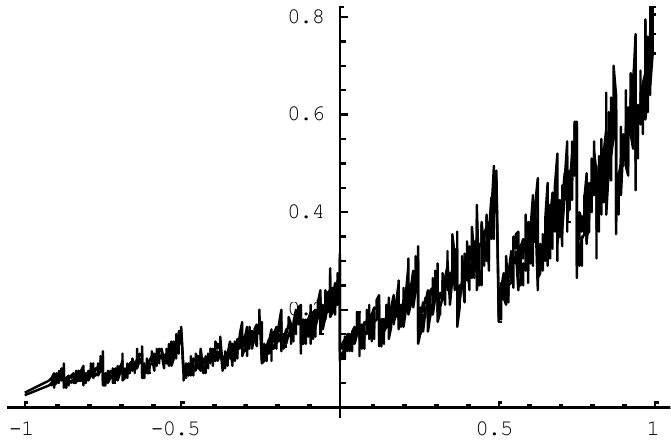} }}%
	\caption{ (Left) The graphs of  $z_3(x)$ and $z_4(x)$  for $\gamma=1.88$.
	(Right) graphs of $z_5(x)$ and $z_6(x)$  for $\gamma=1.88$.}%
	\label{fig8}%
\end{figure}

\begin{figure}[h!]
	\centering
	\subfloat{{\includegraphics[scale=0.661,angle=0]{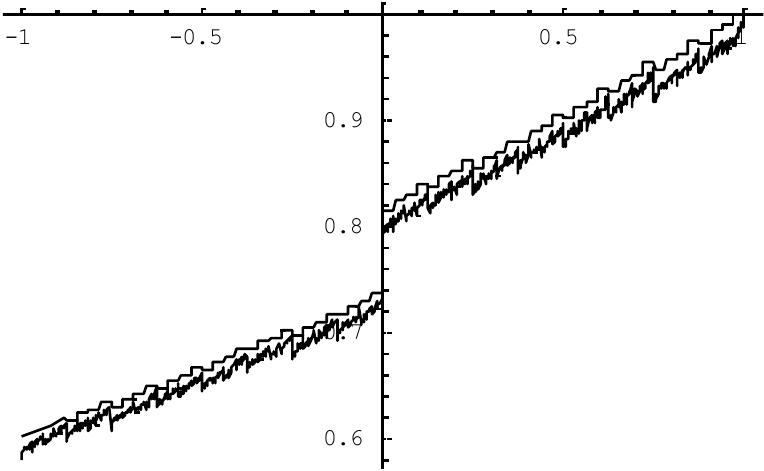} }}%
	\qquad
	\subfloat{\includegraphics[scale=0.661,angle=0]{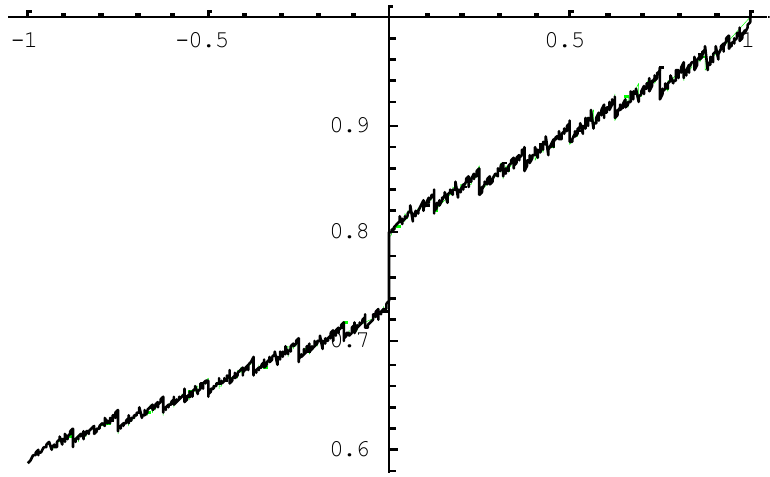} }%
    \caption{The graphs of the eigenfunction $\varphi$
     and  of the approximation obtained via the involution kernel when $\gamma=3.3$.
     The picture on the left side was obtained via thermodynamic limit and  the one on the right side was obtained
     via simulation of a Bernoulli process - the picture on green is the graph of the numerical analytical
     approximation of the explicit
     eigenfunction and the picture on black is the approximation obtained via involution kernel.}
   \label{fig10}
\end{figure}

We recall that in Section \ref{inv} the expression (\ref{expinv})
describes the eigenfunction in terms of the eigenprobability (explicitly known when $\gamma>2$)
and the involution kernel (explicitly known when $\gamma>2$).
In Figure \ref{fig10} we use the parameter $\gamma=3.3$ and plotted the graphs
of the continuous eigenfunction and the graph of an approximation obtained using
the expression of the eigenprobability and the involution kernel - in the case of the potential $A(x)=
\sum_{j=1}^\infty j^{-\gamma}x_j$.
In this case, in order to compute the expression of the
eigenfunction via this method we need a  "numerical approximation" of
the eigenprobability. A natural procedure to do that is
via thermodynamic limit as described on Section 3.1 in \cite{sarig2009lecture} or on
Section 8 of \cite{cioletti2014interactions}.
We took preimages at level 6 of the point $1^\infty$ to generate the
left hand side  picture of Figure  \ref{fig10}.

The eigenprobability for the  potential $A(x)=
\sum_{j=1}^\infty j^{-\gamma}x_j$ is a non-stationary  independent Bernoulli probability explicitly known.
Then,
we can simulate - with better precision - the eigenprobability via the following procedure:
take  a sequence  of flipping coins
with different
probabilities. In this case the result is presented on the right hand side of figure \ref{fig10}.

\end{document}